\documentclass[a4paper]{amsart}
\usepackage[utf8]{inputenc}
\usepackage[T1]{fontenc}
\usepackage{ae,aecompl}
\usepackage{enumerate}
\usepackage{ifpdf}
\ifpdf
\fi
\usepackage{array}
\usepackage{graphicx}

\title[Some relational structures and their associated algebras]
{Some relational structures with polynomial growth and their
  associated algebras II.\\
  Finite generation.}

\author{Maurice Pouzet}

\thanks{This work was done under the auspices of the Intas programme
  Universal algebra and Lattice theory, and supported by CMCU
  Franco-Tunisien "Outils math\'ematiques pour l'informatique". Some
  of the results (invariants of permutation groupoids) were announced
  in an extended abstract for
  FPSAC'05~\cite{Pouzet_Thiery.AgeAlgebra.2005}. An early short
  version of this paper was presented at the ROGICS'08 conference, May
  12-17, Mahdia (Tunisia)~\cite{Pouzet_Thiery.AgeAlgebra2alpha}. The
  authors would like to express their gratitude to the organizers of
  the latter, Professors Y.~Boudabbous and N.~Zaguia.}

\address{Univ. Lyon, Universit\'e Claude-Bernard Lyon1, CNRS UMR 5208, Institut Camille Jordan,  43 bd. 11 Novembre 1918, 69622 Villeurbanne Cedex, France and Mathematics \& Statistics Department, University of Calgary, Calgary, Alberta, Canada T2N 1N4}
\email{pouzet@univ-lyon1.fr, maurice.pouzet@ucalgary.ca }
\author{Nicolas M. Thi\'ery}

\address{Univ Paris-Sud, Laboratoire de Recherche en Informatique,
  Orsay, F-91405; CNRS, Orsay, F-91405}
\email{Nicolas.Thiery@u-psud.fr}

\usepackage{color}
\usepackage{ifthen}
\newboolean{draft}
\setboolean{draft}{true}
\ifdraft
\newcommand{\TODO}[2][To do: ]{\textcolor{red}{\textbf{#1#2}}}
\else
\newcommand{\TODO}[2][]{}
\fi

\usepackage{amsfonts,amsthm,amssymb}
\usepackage{hyperref}

\usepackage{enumerate}
\usepackage{enumerate}

\newcommand{\suchthat}{{\ :\ }}
\newcommand{\leaf}{\circ}

\newcommand{\sg}{{\mathfrak S}}

\newcommand{\bijection}{\hookrightarrow\!\!\!\!\!\rightarrow}

\newcommand{\Id}{{\operatorname{Id}}}
\newcommand{\kernel}{{\operatorname{kernel}}}
\newcommand{\dom}{{\operatorname{dom}}}
\newcommand{\im}{{\operatorname{im}}}

\newcommand{\id}{{\operatorname{id}}}

\newcommand{\loc}{{\operatorname{Loc}}}
\newcommand{\qsym}{{\operatorname{QSym}}}
\newcommand{\sym}{{\operatorname{Sym}}}
\newcommand{\lm}{{\operatorname{lm}}}
\newcommand{\hilbert}{{\mathcal H}}
\newcommand{\reynolds}{\mathcal R}
\newcommand{\age}{{\mathcal A}}
\newcommand{\agealgebra}{{\K.\mathcal A}}
\newcommand{\profile}{\varphi}
\newcommand{\poset}{{\mathcal P}}
\newcommand{\lex}{{\operatorname{lex}}}
\newcommand{\setalgebra}[1][E]{{\K^{[#1]^{<\omega}}}}
\newcommand{\dvect}{{\mathbf{d}}}
\newcommand{\x}{{X}}

\newcommand{\R}{\mathbb{R}}
\newcommand{\K}{\mathbb{K}}
\newcommand{\N}{\mathbb{N}}

\newcommand{\Z}{\mathbb{Z}}
\newcommand{\Q}{\mathbb{Q}}

\newcommand{\type}[1]{\tau(#1)}
\newcommand{\ideal}[1][I]{\mathcal #1}

\newtheorem{theorem}{Theorem}[section]

\newtheorem{lemma}[theorem]{Lemma}

\newtheorem{proposition}[theorem]{Proposition} 
\newtheorem{corollary}[theorem]{Corollary} 
 
\theoremstyle{definition}
\newtheorem{definition}[theorem]{Definition}
\newtheorem{problem}[theorem]{Problem}

\newtheorem{fact}[theorem]{Fact}
\theoremstyle{remark}
\newtheorem{remark}[theorem]{Remark}
\newtheorem{remarks}[theorem]{Remarks}

\newtheorem{conditions}[theorem]{Conditions}
\newtheorem{condition}[theorem]{Condition}
\newtheorem{example}[theorem]{Example}
\newtheorem{examples}[theorem]{Examples}

\def\centerpicture #1 by #2 (#3){\leavevmode
        \vbox to #2{
        \hrule width #1 height 0pt depth 0pt
        \vfill
        \special{pictfile #3}}}

\makeatletter
\newskip\@bigflushglue \@bigflushglue = -100pt plus 1fil

\def\bigcentering{\let\\\@centercr\rightskip\@bigflushglue%
\leftskip\@bigflushglue
\parindent\z@\parfillskip\z@skip}

\makeatother

\begin{document}
\maketitle

\begin{abstract}
  The \emph{profile} of a relational structure $R$ is the function
  $\profile_R$ which counts for every integer $n$ the number, possibly
  infinite, $\profile_R(n)$ of substructures of $R$ induced on the
  $n$-element subsets, isomorphic substructures being identified.  If
  $\profile_R$ takes only finite values, this is the Hilbert function
  of a graded algebra associated with $R$, the \emph{age algebra}
  $\agealgebra(R)$, introduced by P.~J.~Cameron.

  In a previous paper,
  we studied the relationship between the properties of a relational
  structure and those of their algebra, particularly when the
  relational structure $R$ admits a finite monomorphic
  decomposition. This setting still encompasses well-studied graded
  commutative algebras like invariant rings of finite permutation
  groups, or the rings of quasi-symmetric polynomials.

  In this paper,
  we investigate how far the well know algebraic properties of those
  rings extend to age algebras. The main result is a combinatorial
  characterization of when the age algebra is finitely generated. In the
  special case of tournaments, we show that the age algebra is
  finitely generated if and only if the profile is bounded.
  We explore the Cohen-Macaulay property in the special case of
  invariants of permutation groupoids. Finally, we exhibit sufficient
  conditions on the relational structure that make naturally the age
  algebra into a Hopf algebra.

\end{abstract}

\section*{Introduction}

In~\cite[Section 2]{Cameron.1981} Cameron defined the \emph{orbit
  algebra} of a permutation group $G$ acting on an infinite set $E$;
by design, the Hilbert function $h_{\agealgebra(G)}$ of this graded
commutative algebra coincides with the \emph{orbital profile} of $G$,
namely the function that counts, for every $n$, the number
$\profile_G(n)$ of orbits of $G$ acting the finite subsets of size $n$
of $E$. The main motivation was to study properties of orbital
profiles, and in particular a phenomenon of jumps in the possible
growth rates.

Similar phenomenon had been observed in the more general context of
relational structures (permutation groups being in correspondence with
homogeneous relational structures). There, the profile of a relational
structure $R$ on $E$ counts, for every integer $n$, the number
$\profile_R(n)$ of substructures of $R$ induced on the $n$-element
subsets of $E$, isomorphic substructures being identified. In
\cite{Cameron.1997}, Cameron proposed to generalize the approach,
defining the age algebra of a relational structure. Familiar algebras
like invariant rings of finite permutation groups, algebras of
quasi-symmetric polynomials~\cite{Gessel.QSym.1984} or the shuffle
algebra over a finite alphabet can be realized as such age algebras.

As a follow up to \cite{Pouzet_Thiery.AgeAlgebra1}, this paper
investigates relationships between combinatorial properties of a
relational structure $R$ and algebraic properties of its age algebra
$\agealgebra(R)$. Specifically, we consider the following conditions:
\begin{conditions}
  \label{conditions}
   \begin{enumerate}
   \item[(BP)] the profile is bounded above by a polynomial;
     \label{condition.bounded_polynomial}
   \item[(QP)] the profile is eventually a quasi-polynomial;
     equivalently its generating series is of the form:
     \begin{equation*}
       \frac{P(Z)}{(1-Z^{n_1})(1-Z^{n_2})\cdots(1-Z^{n_k})}\ ,
     \end{equation*}
     where $n_1\leq \cdots \leq n_k$ and $P(Z)\in\Z[Z]$;
     \label{condition.quasipolynomial}
   \item[(QPP)] same as (QP) with $P\in \N[Z]$;
     \label{condition.positivequasipolynomial}
   \item[(AP)] the profile is asymptotically equivalent to a
     polynomial;
     \label{condition.assymptoticallypolynomial}
   \item[(FG)] the age algebra is finitely generated;
     \label{condition.finitely_generated}
   \item[(CM)] the age algebra is Cohen-Macaulay;
     \label{condition.cohen_macaulay}
  \end{enumerate}
\end{conditions}
We also consider the following condition:
\begin{condition}
  \label{condition.hopf}
  \begin{enumerate}
  \item[(H)] the age algebra is a graded Hopf algebra; in particular
    it is free.
  \end{enumerate}
\end{condition}

Let us review what is known, starting with the obvious or well know
relations between those conditions:
\begin{itemize}
\item (QP) $\Longrightarrow$ (BP), (AP) $\Longrightarrow$ (BP);
\item (QP) $\Longrightarrow$ (AP), using that the profile is non
  decreasing (Pouzet, \cite[ex. 8 p. 113]{Fraisse.CLM1.1971} for
  relational structures; Cameron, \cite[Theorem 2.2]{Cameron.1976} for
  permutation groups);
\item The two conditions of (QP) are equivalent as a straightforward
  consequence of~\cite[Proposition
  4.4.1]{Stanley.1999.EnumerativeCombinatorics1};
\item (FG) $\Longrightarrow$ (QP) by a general property of
  graded commutative algebras (see e.g.~\cite[Chapter 9, \S 2]{Cox_al.IVA});
\item (CM) $\Longrightarrow$ (QPP);
\item (CM) $\Longrightarrow$ (FG).
\end{itemize}

For the examples mentioned earlier, all of Conditions~\ref{conditions} are
equivalent. This is not an isolated phenomenon. Let us consider the case
of a permutation group, or equivalently of an homogeneous relational
structure. Cameron conjectured (BP)
$\Longrightarrow$ (AP)~\cite[p.~69]{Cameron.1990} and Macpherson asked
whether (BP) $\Longrightarrow$ (AP)~\cite[p. 286]{Macpherson.1985}.
Justine Falque and the second author recently provided a positive
answer; in fact, all of Conditions~\ref{conditions} are
equivalent~\cite{FalqueThiery.Macpherson.FPSAC}.

In \cite{Pouzet_Thiery.AgeAlgebra1}, we conjectured that, when the
kernel of $R$ is finite, the profile $\varphi_R$ of a relational
structure $R$ is eventually a quasi-polynomial whenever $\varphi_R$ is
bounded by some polynomial (that is (BP) $\Longrightarrow$ (AP)). We
then introduced the notion of monomorphic decomposition, restricted
ourselves to the case of relational structure admitting a finite
monomorphic decomposition and proved our conjecture there.

In this paper, we proceed in the same setting, and investigate the
conditions (FG), (CM), and (H). This setting encompasses invariant
rings of finite permutation groups and rings of quasi-symmetric
polynomials for which conditions (FG) and (CM) hold; the later fact is
a theorem of Garsia and Wallach~\cite{Garsia_Wallach.2003}. For other
examples, (FG) and (CM) fail. Our main result is a combinatorial
characterization of when (FG) holds
(Theorem~\ref{theorem.finiteGeneration}).

In Section~\ref{section.ageAlgebra}, we briefly review relational
structures, their orbit algebras and monomorphic decompositions. We
refer to \cite{Pouzet_Thiery.AgeAlgebra1} for a detailed approach. In
addition, we mention there a relationship between the order properties
of an age and properties of the ideals of the age algebra.

Section~\ref{section.finiteGeneration} is devoted to our main
theorem (Theorem~\ref{theorem.finiteGeneration}). We start by giving
the key ideas on an example, and proceed with the general proof.
With the help of~\cite{Boudabbous_Pouzet.2009},
we show that the age algebra of a tournament is finitely generated if
and only if the profile of the tournament is bounded
(Theorem~\ref{theorem.finiteGenerationTournament}). Indeed, if an age
algebra is finitely generated, the profile is bounded above by a
polynomial and according to ~\cite{Boudabbous_Pouzet.2009},
tournaments with profile bounded by a polynomial have a finite
monomorphic decomposition (meaning simply that these tournaments are
lexicographical sums of acyclic tournaments indexed by a finite
tournament) and our characterization applies.

In Section~\ref{section.invariant.groupoids}, we further restrict the
setting to \emph{invariant rings of permutation groupoids}, defined as
age algebras of some appropriate relational structures
(Section~\ref{section.invariant.groupoids}). This setting provides a
very tight generalization of invariant rings of permutation groups
which still includes quasi-symmetric polynomials. We analyze in
details which properties of invariant rings of permutation groups
carry over -- or not -- to permutation groupoids (cf.
Propositions~\ref{proposition.groupoid.derivation}
and~\ref{proposition.groupoids.reynolds}, and
Theorems~\ref{theorem.finiteGeneration} and~\ref{theorem.SAGBI}). To
this end, we use in particular techniques
from~\cite{Garsia_Stanton.1984}.

Finally, in Section~\ref{section.hopf}, we give some sufficient
conditions on the relational structure to endow the age algebra with a
further structure of (coassociative) Hopf algebra, and recover several
classical Hopf algebras. The age algebra is then a free algebra, which
imposes a very rigid form for the Hilbert series and thus for the
profile.

We conclude this introduction with general comments and perspectives.

For orbit algebras, our main theorem does not bring new insight;
indeed, by Theorem~\ref{theorem.orbit_algebra_finite_monomorphic}, an
orbit algebra whose homogeneous structure admits a finite monomorphic
decomposition is isomorphic to the invariant rings of a finite
permutation group (or straightforward quotient thereof); the latter is
well known to be finitely generated.

There are other classes of structures for which polynomially bounded
profile amounts to the existence of a finite monomorphic decomposition
(e.g. permutations~\cite{Monteil_Pouzet.2008} and ordered
graphs~\cite{Balogh1,Balogh3}) and for which we may use our
characterization.

\begin{problem}
  For relational structures admitting a finite monomorphic
  decomposition, characterize combinatorially when the age algebra is
  Cohen-Macaulay. This would provide an alternate proof of Garsia and
  Wallach's theorem for quasi-symmetric functions.
\end{problem}

\begin{problem}
  For general relational structures, characterize combinatorially when
  the age algebra is finitely generated. The remaining open case is
  when the minimal monomorphic decomposition has infinitely many
  blocs, a finite number of which being infinite. Examples~A.9
  and~A.10 of~\cite{Pouzet_Thiery.AgeAlgebra1} show that, in this case
  and even just for graphs, the age algebra can be finitely generated,
  or not.
\end{problem}

Ultimately the notion of age algebra may just not be quite right, and
should be adapted to ensure that all of Conditions~\ref{conditions} are equivalent:
\begin{problem}
  Devise some \emph{sensible} alternative graded algebra structure on
  $\agealgebra(R)$ which is finitely generated whenever the profile
  is bounded above by a polynomial.
\end{problem}
\begin{problem}
  Devise some \emph{sensible} alternative graded algebra structure on
  $\agealgebra(R)$ which is Cohen-Macaulay whenever the Hilbert series
  has the appropriate form (by Proposition~4
  of~\cite{Bogvad_Meyer.2003} such an algebra always exists).
\end{problem}

Let $R$ be a relational structure on a set $E$. It induces an
equivalence relation on the finite subsets of $E$ by setting
$A\sim_R B$ whenever the structures induced by $R$ on $A$ and $B$
respectively are isomorphic. This equivalence relation is
\emph{hereditary}:
\begin{displaymath}
  |A|=|B| \qquad \text{ and } \qquad
  |\{X\subset A\suchthat X\sim_R C\}| = |\{X\subset B\suchthat X\sim_R C\}\,|
\end{displaymath}
whenever $A, B$, and $C$ are finite subsets of $E$ such that
$A\sim_R B$ (hereditary equivalence relations were introduced
in~\cite{Pouzet_Rosenberg.1986}; see
also~\cite{Buchwalder.2009.Thesis}. The definition of age algebra
extends straightforwardly to hereditary equivalence relations.
\begin{problem}
  Generalize the results of this paper to hereditary equivalence
  relations and the corresponding  algebras.
\end{problem}

\section{On the profile  and age algebra of a relational structure}
\label{section.ageAlgebra}

\subsection{Relational structures and their monomorphic decompositions}

A \emph{relational structure} is a pair $R:=(E,(\rho_{i})_{i\in I})$
made of a set $E$ and a family of $m_i$-ary relations $\rho_i$ on $E$.
We denote by $R_{\restriction A}$, that we call \emph{restriction} of $R$ to $A$,  the substructure  induced by
$R$ on a subset $A$ of $E$. We consider these substructures up to
isomorphism. If needed,   we consider  \emph{isomorphic types},  objects
associated to relational structures in such a way that the types
$\tau(S_1)$ and $\tau (S_2)$ of two relational structures are equal if
and only if the two relational structures are isomorphic. In our case,
we may  identify the isomorphic type $\tau (R_{\restriction A})$ of the
substructure of $R$ induced on a finite subset $A$ of $E$ to its
\emph{orbit} $\tau (A):= \{A'\subseteq E: R_{\restriction_A} \;
\text {is isomorphic to}\; R_{\restriction_A}\}$.

\label{section.finiteMonomorphicDecomposition}

Let $R$ be a relational structure on a set $E$.  A subset $B$ of $E$
is a \emph{monomorphic part} of $R$ if for every integer $n$ and every
pair $A, A' $ of $n$-element subsets of $E$ the induced structures on
$A$ and $A'$ are isomorphic whenever $A\setminus B=A'\setminus B$. 
A partition of $E$ into monomorphic parts is a \emph{monomorphic decomposition} of $R$. 

The crucial property of such a  partition is given in Fact~\ref{factdecomp}:

Let $(E_x)_{x\in X}$ be a set partition of $E$. Set $X_\infty:=\{x\in
X \suchthat |E_x|=\infty\}$; for a finite subset $A$ of $E$,
set $d_x(A):=|A\cap E_x|$, and denote by $\dvect(A):=(d_x(A))_{x\in
  X}$ the statistics of intersection sizes. 
  
\begin{fact}\label{factdecomp}\noindent \emph{ The partition    $(E_x)_{x\in X}$ is a
monomorphic decomposition  of $R$ if and only if  the induced structures on
two finite subsets $A$ and $A'$ of $E$ are isomorphic whenever
$\dvect(A)=\dvect(A')$}.  
\end{fact}

As stated in~\cite{Pouzet_Thiery.AgeAlgebra1}[Proposition~2.12], each monomorphic part of $E$ is included into a maximal one w.r.t. inclusion and these maximal parts form a monomorphic decomposition of $R$.  Hence, every monomorphic decomposition is finer than the decomposition into  maximal parts. We call this partition the \emph{minimal monomorphic decomposition}.
The \emph{monomorphic dimension} of $R$ is the number of infinite
parts in its minimal monomorphic decomposition.

Revisiting these notions, Oudrar and Pouzet (see~\cite[p.~168,
\S~7.2.5]{Oudrar.2015}, \cite{Oudrar_Pouzet.2013,Pouzet_SiKaddour.2016})
define this partition in a direct way as follows:
Say that two elements $x$ and  $y$  of $E$ are \emph{equivalent} and set $x\simeq_Ry$ if for every finite subset $F$ of   $E\setminus \{x,y\}$, the restrictions of $R$ to   $\{x\}\cup F$  and $\{y\}\cup F$ are isomorphic.

\begin{lemma}
  The relation $\simeq_R$ is an equivalence relation. Furthermore, the
  equivalence classes of $\simeq_R$ are the maximal parts of $R$,
  hence they form the minimal monomorphic decomposition.
\end{lemma}

As an immediate consequence, we obtain  the following result (occurring as a part of Corollary~2.13 of~\cite{Pouzet_Thiery.AgeAlgebra1}).  

\begin{corollary} \label{cor:automorphism} Every automorphism of $R$ induces a permutation of the maximal parts of $R$. 
\end{corollary}

\subsection{Profile and  age algebra of a relational structure}
The \emph{age} of a relational structure $R:=(E,(\rho_{i})_{i\in I})$ is the set $\age:=\age(R)$ of finite substructures of $R$,
isomorphic substructures being identified. This set was introduced by
Fra{\"\i}ss{\'e} (see~\cite{Fraisse.TR.2000}). The \emph{kernel} of
$R$ is the set $\kernel(R):= \{x\in E: \age(R_{\restriction E} )\not =
\age (R)\}$. We say that $\age(R)$ is \emph{inexhaustible} and $R$ is
\emph{age-inexhaustible} if two members of $\age(R)$ can be embedded in a third with disjoint images. This condition amounts to say that the kernel of $R$ is empty.

The \emph{profile} of $R$ is the function $\profile_R$ which counts
for every integer $n$ the number $\profile_R(n)$ of substructures of
$R$ induced on the $n$-element subsets, isomorphic substructures being
identified.  Clearly, this function only depends  upon the age of $R$. We recall that 
the profile of an infinite relational structure is
non-decreasing. Furthermore, provided that some mild conditions hold, namely either the signature $\mu : =
(m_i)_{i \in I}$ is bounded or the kernel $\kernel(R)$ of $R$ is finite, there are
jumps in the behavior of the profile: the
 growth of $\profile_R$ is either  polynomial  or as fast as
every polynomial 
\cite{pouzet.tr.1978, Pouzet.2006.SurveyProfile} and  \cite{Balogh1, Balogh3, Klazar.2010} for independent  developments on this theme. 

{\bf N.B.: in the sequel, and for the sake of simplicity of
  exposition, we always make the assumption that $\profile_R(n)$ is
  finite for all $n$}. This holds as soon as $I$ is finite.

Let $\K$ be a field of characteristic $0$. Cameron associates a graded algebra $\K.\age (R)$ to each
relational structure $R$~\cite{Cameron.1997}. This  algebra $\K.\age (R)$ depends only upon the age of $R$. 
Its main feature is that
its Hilbert function coincides with the profile of $R$ as long as it
takes only finite values. 

The graded algebra $\K.\age (R)$ is the direct sum
$\bigoplus_{n<\omega} \K.\age (R)_{n}$, where $\K.\age (R)_{n}$ is the
set of $\K$-linear combinations of elements of $\age (R)_{n}$, the set
of substructures of $R$ induced on the $n$-element subsets of $E$
considered up to isomorphy.

 Multiplication is defined by the rule that if $\tau_1 \in  \age (R)_{n}$, $\tau_2 \in \age (R)_{m}$ then 
\begin{equation}\label{eq:product}\tau_1.\tau_2:= \sum_{\tau \in \age (R)_{n+m}}  c^{\tau}_{\tau_1, \tau_2}\tau
\end{equation}
 where \begin{equation}
  c_{\tau_1,\tau_2}^\tau := |\{ (A_1,A_2) \suchthat A_1\uplus A_2=A,
  \type{A_1}=\tau_1, \type{A_2}=\tau_2 \text { and } \tau (A)= \tau \}|\,.
\end{equation}

We recall two results: 

Let $e$ be the sum of isomorphic types of the one-element restrictions of $R$ (we can identify it to $\sum_{a\in E} \{a\}$). Let $U$ be
the graded free algebra $\K[e]=\bigoplus_{n=0}^\infty \K e^n$.

\begin{theorem}
  \label{theorem.cameron}\cite{Cameron.1997}
  If $R$ is infinite then $e$ is not a zero-divisor; namely for any
  $u\in \agealgebra(R)$, $e u=0$ if and only if $u=0$.
\end{theorem}
This result implies that $\profile_R$ is non decreasing. Indeed, the
image of a basis of the vector space $\agealgebra(R)_n$ under
multiplication by $e$ is a linearly independent subset of
$\agealgebra(R)_{n+1}$.

 We recall the following result: 

\begin{theorem}\label{thm:non-zero} \cite{Pouzet.2008.IntegralDomain} If  $R$ is age-inexhaustible then the age algebra $\K.\age (R)$ has no non-zero
divisor.
\end{theorem}

In the sequel, we give some general properties relating relational
structures and algebras.

\subsection{Operations on relational structures and age algebras}

In this section and in Section~\ref{ageAlgebraPoly}, a useful
technical device is to embed the age algebra $\agealgebra(R)$ of a
relational structure $R$ in a larger algebra, the set algebra, whose
definition we recall now. 

Let $\K$ be a field of characteristic $0$,
and $E$ be a set. For $n\geq 0$, denote by $[E]^n$ the set of the
subsets of $E$ of size $n$, and let $\K^{[E]^{n}}$ be the vector space
of maps $f: [E]^n \rightarrow \K$. The \emph{set algebra} is the
graded connected commutative algebra $\setalgebra:=\bigoplus_n
\K^{[E]^{n}}$, where the product of $f: [E]^m\rightarrow \K$ and $g:
[E]^n\rightarrow \K$ is defined as $fg: [E]^{m+n}\rightarrow \K$ such
that:
\begin{equation}
  (fg)(A):= \sum_{(A_1,A_2) \suchthat  A=A_1\uplus A_2}  f(A_1)g(A_2) \,.
\end{equation}
Identifying a set $S$ with its characteristic function $\chi_S$,
elements of the set algebra can be thought as (possibly infinite but of
bounded degree) linear combination of sets; the unit is the empty set,
and the product of two sets is their disjoint union, or $0$ if their
intersection is non trivial.

The desired embedding of $\agealgebra(R)$ into $\setalgebra$ is
obtained by mapping an isomorphism type $\tau$ to its orbit sum
$\sum_{A\suchthat\tau(A)=\tau} A$.

\begin{remark}
  The original definition of the age algebra, which we used above,
  requires the assumption that the profile is finite. See e.g.
  Example~\ref{example.infinite_profile} for what can go wrong otherwise.
  The set algebra
  offers a natural frame to formulate an equivalent definition that
  extends beyond this assumption and further to hereditary equivalence
  relations (this is for example the presentation adopted
  in~\cite{Pouzet.2008.IntegralDomain}).

  Namely, consider a relational structure $R$ on a set $E$ with
  profile not necessarily finite, or more generally a hereditary
  equivalence $\equiv$ on the finite subsets of $E$. Say that a map
  $f: [E]^n \rightarrow \K$ is \emph{invariant} if it is constant on
  each equivalence class on $[E]^n$ induced by $R$ or $\equiv$ (e.g.
  $f(A)=f(A')$ whenever the restrictions $R_{\restriction A}$ and
  $R_{\restriction A'}$ are isomorphic). Consider the space
  $\agealgebra=\bigoplus_n \agealgebra_n$, where $\agealgebra_n$ is
  the set of all invariant maps for a given $n$. Observe that
  $\agealgebra$ is a subalgebra of the set algebra, and call it the
  age algebra.
\end{remark}

\begin{proposition}
  \label{proposition.constructions}
  \begin{enumerate}[(i)]
  \item Let $R:=(E, (\rho_i)_{i\in I})$ be a relational structure,  
    $E'$ be a subset of $E$ and  $R':= R_{\restriction E'}$ be the structure induced by $R$
    on $E'$. Then, $\agealgebra(R')$ is a quotient of
    $\agealgebra(R)$.
  \item Let $R:=(E, (\rho_i)_{i\in I})$ be a relational structure,  
    $I'$ be a subset of $I$, and  $R':=(E, (\rho_i)_{i\in I'})$. Then, $\agealgebra(R')$ is a
    subalgebra of $\agealgebra(R)$.
  \item Let $R:=(E, (\rho_i)_{i\in I})$ and $R':=(E', (\rho'_i)_{i\in
      I'})$ be two relational structures on disjoint subsets, and
    define their \emph{direct sum} as the relational structure
    $R\oplus R' := (E\cup E', (\rho_i)_{i\in I}, (\rho'_i)_{i\in
      I'})$. If needed, add an appropriate unary relation to ensure
    that singletons of $E$ and $E'$ are not isomorphic. Then,
    $\agealgebra(R\oplus R')$ is isomorphic to the tensor product $\agealgebra(R)
    \otimes \agealgebra(R')$. Furthermore, if $(E_i)_i$ and $(E'_j)_j$
    are (minimal) monomorphic decompositions of respectively $R$ and
    $R'$, then $(E_i)_i\cup (E'_j)_j$ form a (minimal) monomorphic
    decomposition of $R\oplus R'$.

  \end{enumerate}
\end{proposition}
\begin{proof}
  (i) At the level of the set algebra, the vector space spanned by the
  sets which are not subsets of $E'$ is an ideal; so the linear map
  $\phi$ which kills those sets is an algebra morphism. Furthermore
  looking at the image of basis elements shows that $\agealgebra(R')$
  is the image by $\phi$ of $\agealgebra(R)$.

  (ii) Each isomorphism class for $R$ splits into one or more
  isomorphism class(es) for $R'$; hence each basis element for
  $\agealgebra(R)$ is accordingly the sum of one or more basis
  element(s) for $\agealgebra(R')$.

  (iii) Identify each subset $A$ of $E\cup E'$ with the element
  $(A\cap E) \otimes (A\cap E')$ of the tensor product and check that a basis element for
  $\agealgebra(R\otimes R')$ is the tensor product of a basis element
  for $\agealgebra(R)$ by a basis element of $\agealgebra(R')$. The
  resulting Hilbert series is the product of the Hilbert series for
  $R$ and $R'$.
\end{proof}

Let $R=(E, (\rho_i)_{i\in I})$ and $R'=(E', (\rho'_i)_{i\in I'})$ be
two relational structures. The \emph{wreath product} of $R$ and $R'$
is the relational structure
$R\wr R':=\left(E\times E', (\tilde \rho_i)_{i\in I}, (\tilde
  \rho'_i)_{i\in I'}\right)$, where
\begin{displaymath}
  \tilde\rho_i\left( (e_1,e'_1), \dots, (e_k,e'_k) \right)%
  \text{ if and only if }%
  \rho_i(e_1,\dots,e_k) \text{ and } e'_1=\dots=e'_k\,,
\end{displaymath}
and
\begin{displaymath}
  \tilde\rho_i\left( (e_1,e'_1), \dots, (e_k,e'_k) \right) %
  \text{ if and only if }%
  \rho'_i(e'_1,\dots,e'_k)\,.
\end{displaymath}

\begin{proposition}
  If $(E_i)_i$ and $(E'_j)_j$ are monomorphic decompositions of
  respectively $R$ and $R'$, then $(E_i\times E'_j)_{i,j}$ forms a
  monomorphic decomposition of $R\wr R'$. In particular, if $R$ is
  monomorphic, then the $E_x := E \times \{x\}$ for $x\in E'$ form a
  monomorphic decomposition of $R\wr R'$.
\end{proposition}
\begin{proof}
  Direct consequence of Proposition~\ref{proposition.constructions} (ii) and (iii).
\end{proof}

Let $R$ be a relational structure. A \emph{local isomorphism} of $R$
is any isomorphism from a restriction of $R$ to another restriction.
We denote by $\loc(R)$ the set of local isomorphisms of $R$. Endowed
with the partial composition product, $\loc(R)$ becomes a monoid (see
e.g. Section~\ref{section.groupoids}). The wreath product of two such
monoids can be defined in the same vein as for permutation groups:
$\loc(R)$ acts independently within each $E_x$ while $\loc(R')$ acts
globally on the $(E_x)_{x\in E'}$.

\begin{remark}
  $\loc(R\wr R')$ includes the wreath product $\loc(R)\wr \loc(R')$;
  this inclusion may be strict: take the wreath product of a
  $2$-antichain with itself.
\end{remark}

\subsection{Initial  segments of an age and ideals of a ring}

Final segments play for posets the same role than ideals for rings. In
this section, we order the age of a relational structure by
embeddability: if $\tau_1, \tau_2 \in \age(R)$, we set
$\tau_1\leq \tau_2$ if $\tau_1$ is the type of a structure induced on
some structure of type $\tau_2$. This order is up-directed: for any
$\tau_1,\tau_2\in \age(R)$ there exists $\tau\in\age(R)$ such that
$\tau_1, \tau_2\leq \tau$.
An \emph{initial segment} of $\age(R)$ is any subset $F$ of $\age(R)$
such that $\tau_2\in F$ implies $\tau_1\in F$ whenever
$\tau_1\leq \tau_2$.

We describe briefly the correspondence between initial segments of an age  and ideals of the age algebra. %

\begin{theorem}\label{thm:final-ideal} Let $\age:=\age(R)$ be the  age of a relational structure $R$ and $\agealgebra$ be its age algebra.
  Recall that we assume that $R$ has finite profile. If $\age'$ is an initial segment of $\age$ then:
 \begin{enumerate} [{(i)}]
 \item the vector subspace $J:= \K.(\age\setminus \age')$ spanned by $\age\setminus \age'$ is an ideal of $\K.\age $. Moreover, the quotient of $\agealgebra$ by $J$ is  a ring isomorphic to  the ring $\agealgebra'$. 
 \item if $J$  is irreducible then $\age'$ is a subage of $\age$;
 \item $J$ is a prime ideal if and only if $\age'$ is an inexhaustible age.
 \end{enumerate}
\end{theorem}
\begin{proof}

$(i)$ Since $J$ is a subspace of $\K.\age $, it suffices to show that:
\begin{equation}\label{eqideal}
  P\in \age\setminus \age'
  \; \text { and }\;
  Q\in \age
  \;\text{ implies }\;
  P.Q\in J\,.
\end{equation}

We have $P.Q:=\sum_{R= P\cup Q} R$. Each member $R$ of this sum embeds $P$; since $\age'$ is an initial segment of $\age$ and $P\in \age\setminus \age'$, it follows that  $R\in  \age\setminus \age'$ proving that $P.Q\in J$. 

$(ii)$ For convenience, we consider here complements of initial
segments, that is \emph{final segments}.
Let  $F(\age)$ be the set of final segments of $ \age$ and $\Id(\agealgebra)$ be the set of ideals  of $\agealgebra$, these sets being ordered by inclusion.  Let $\varphi: F(\age)\rightarrow \Id(\agealgebra)$ defined by setting  $\varphi(F):= \K.F$ if $F\not =\emptyset$ and $\varphi(F):= \{0\}$ otherwise. As shown in $(i)$ this map is well-defined. It preserves arbitrary joins and meets that is:
\begin{equation}\label{eqjoin}
\K.\bigcup_{i\in I}F_i = \sum_{i\in I}\K.F_i\,;
\end{equation}
\begin{equation}\label{eqmeet}
\K.\bigcap_{i\in I}F_i = \bigcap_{i\in I} \K.F_i\,.
\end{equation}
The first equality is obvious. For the second equality, we have
trivially $\K.\bigcap_{i\in I} F_i\subseteq \bigcap_{i\in I} \K.F_i$.
For the converse, let $P\in \bigcap_{i\in I} \K.F_i$. We have
$P= \sum_{j\in I_i}P_{ij}$ with $P_{ij}\in \K.F_i$ for each $i\in I$.
By definition, members of $\age$ are linearly independent, thus
the decomposition of $P$ into a linear sum of members of $\age$
is unique. It follows that $P_{ij}$ is independent of $I_i$ thus
belongs to $\bigcap_{i\in I} F_i$, hence
$P\in \K.\bigcap_{i\in I}F_i$. This completes the proof that the
second equality holds.

The map $\varphi $ restricted to $F(\age)\setminus \{\emptyset\}$ is
one-to-one. The inverse image of a meet-irreducible element of
$\Id(\agealgebra)$ is a meet-irreducible element of $F(\age)$. As it
is well-known, for an arbitrary poset $\poset$, the meet-irreducible
members of $F(\poset)$ are exactly the complements of \emph{ideals} of
$\poset$ (that is up-directed initial segments of $\poset$); in the
case $\poset= \age$ the ideals are the subages of $\age$. This
completes the proof of $(ii)$.

$(iii)$ Suppose that $\age'$ is not an inexhaustible age. Then there are $S'_1, S'_2\in \age'$ such that no $R\in \age'$ extends $S'_1, S'_2$ on the disjoint union of their domains. But then the product $S'_1.S'_2$ is in $\K.(\age\setminus \age')$, proving that $\K(\age\setminus \age')$ is not prime. 
 Conversely, let $P,Q\in \agealgebra$ such that $P.Q\in J:= \K.(\age\setminus \age')$. We may write $P= P_1+P_2$, $Q=Q_1+Q_2$ with $P_1,Q_1\in J$, 
$P_2, Q_2\in \K.\mathcal (A')$. We have $P_2.Q_2\in J$. Since  $\agealgebra'$ is isomorphic to  the quotient $\agealgebra/ J$, the product  $P_2.Q_2$ in $\agealgebra'$ is $0$.  But since $\age'$ is inexhaustible, $\agealgebra'$ has no non-zero divisor by Theorem \ref{thm:non-zero}, hence $P$ or $Q$ belongs to $J$, proving that $J$ is prime. 
\end{proof}
\begin{problem}
  Does the converse of $(ii)$ hold. That is, is
  $J:= \K.(\age\setminus \age')$ meet-irreducible whenever $\age'$ is
  a subage of $\age$?
\end{problem}

Ages are special cases of \emph{hereditary classes of relational
  structures} (see~\cite{Fraisse.TR.2000,Oudrar.2015}). The definition
of age algebra carries over straightforwardly to such classes, and the
correspondence between initial segments and ideals would be best further
explored in this context.

\subsection{Well-quasi-ordering of an age}

The age algebra of a relational structure being graded and connected,
it is finitely generated if and only if it is a Noetherian ring. There
are posets which play a role as important in the theory of ordered
sets as noetherian rings in the theory of ring.  These posets, studied
first by Higman \cite{higman.1952}, are said
\emph{well-quasi-ordered}, in brief \emph{wqo}. Namely, a poset $\poset$ is
\emph{wqo} if the set $F(\poset)$ of final segments of $\poset$ is Noetherian
w.r.t. the inclusion order.

Consider an age $\age(R)$ ordered by embeddability. By (i) of
Theorem~\ref{thm:final-ideal}, $F(\age(R))$ embeds into the
collection of ideals of $\agealgebra(R)$.  Consequently:
\begin{proposition}
  If the age algebra $\agealgebra(R)$ is finitely generated then the
  age of $R$ is well-quasi-ordered by embeddability.
\end {proposition}

The reciprocal does not hold. Indeed, the age of a relational
structure is well-quasi-ordered as soon as the age has polynomial
growth and finite
kernel~\cite{pouzet.tr.1978, Pouzet.2006.SurveyProfile} whereas we have
seen that the age algebra is not necessarily finitely generated in
this case. In fact, well-quasi-ordering of the age does not even imply
that the profile is bounded above by a polynomial: indeed, any age with non polynomially bounded
profile contains a well-quasi-ordered age with the same property (use
Lemma~4.1 of~\cite{Pouzet_Thiery.AgeAlgebra1} and the remark above
it).

\begin{problem}
  Is the profile of a relational structure $R$ bounded by some
  exponential whenever the age $R$ is well-quasi-ordered by
  embeddability?
\end {problem}

\subsection{The age algebra as a subring of a polynomial ring}
\label{ageAlgebraPoly}

\begin{theorem}\label{theorem.ageAlgebraPoly}
  If $R$ has a monomorphic decomposition into finitely many infinite
  parts $(E_i)_{i\in X}$, then the age algebra $\agealgebra(R)$ is
  isomorphic to a subalgebra $\K[X]^R$ of $\K[X]$.

  This may be generalized:
  \begin{itemize}
  \item If some parts are finite, $\agealgebra(R)$ is isomorphic to a
    subalgebra of the quotient ring $\K[X]/(x^{|E_x|+1})$.
  \item If there are infinitely many parts, $\agealgebra(R)$ is
    isomorphic to a subalgebra of $\K[[X]]/(x^{|E_x|+1})$ made of
    series of bounded degree.
  \end{itemize}
\end{theorem}
\begin{proof}
  Let $(E_x)_{x\in X}$ be a finite set partition of $E$ into infinite
  subsets. We first use it to define an embedding of $\K[X]$ into the
  set algebra $\setalgebra$. For an exponent vector
  $\dvect:=(d_x)_{x\in X}\in \N^X$, define the monomial
  $X^\dvect:=\prod_{x\in X} x^{d_x}$ and set $\dvect!:= \prod_{x\in X}
  d_x!$. Set furthermore $O(\dvect) := \{A\subseteq E:
  \dvect(A):=\dvect\}$ and let $\chi_{O(\dvect)}$ be the
  characteristic map of $O(\dvect)\in \setalgebra$, which is best
  interpreted as the (possibly infinite) sum $\sum_{\dvect(A)=\dvect}
  A$. Let $\phi: \K[X] \hookrightarrow \setalgebra$ be defined by
  setting $\phi(\x^\dvect ):= \dvect!  \chi_{O(\dvect)}$. 

  \noindent {\bf Claim.} $\phi$ is a one-to-one morphism of algebras.
  
  Indeed, we have $\chi_{O(\dvect)}\chi_{O(\dvect')} =\frac{(\dvect+\dvect')!}{\dvect!\dvect'!}\chi_{O(\dvect+\dvect')}$.

  For an isomorphic type $\tau$ in the age of $R$, viewed as a subset of $E$, set $\dvect(\tau):=
  \{\dvect(A): A \in \tau\}$ and $\mu (\tau):= \sum_{\dvect\in \dvect(\tau)
  }\frac{1}{\dvect!}\x^\dvect$.  Note that any monomial $X^\dvect$ in
  $\K[X]$ appears in exactly one polynomial $\mu(\tau)$; in particular,
  the later are linearly independent. Finally, let $\K[X]^R$ be the
  subset of $\K[X]$ made of finite linear combinations of polynomials
  of the form $\mu(\tau)$.
  
  \noindent {\bf Claim.}The map $\phi$ induces an isomorphism from $\K[X]^R$
  onto $\agealgebra(R)$.  Indeed, we have:
  \begin{displaymath}
    \phi(\mu(\tau))
    = \sum_{\dvect\in \dvect(\tau)} \phi(\frac{1}{\dvect!} \x^\dvect)
    = \sum_{\dvect\in \dvect(\tau)} \chi_{O(\dvect)}
    = \chi_\tau\ .
  \end{displaymath}

  The generalizations are straightforward.
\end{proof}
In the sequel, when the monomorphic decomposition of a relational
structure is clear from the context, we often use polynomials as a
convenient and compact way of writing elements of its age algebra.

At this stage, a natural question is to characterize the subalgebras
of $\K[X]$ that can be constructed as age algebras of relational
structures with a finite monomorphic decomposition. As suggested in
the introduction, the context can be generalized to algebras of
hereditary equivalence relations. This makes the answer simpler: in
the sequel, we obtain a one-to-one correspondence with subalgebras of
$\K[X]$ defined by ``grouping monomials together''.

Let $\sim$ be an equivalence relation on $E^{<\omega}$, with
equivalence classes $(\tau_i)_{i\in I}$. Consider the subspace
$\agealgebra(\sim)$ of $\setalgebra$ spanned by the $\chi_{\tau_i}$,
where $\tau_i$ ranges through the finite equivalence classes. For
simplicity, we assume that there are finitely equivalence classes for
each size of subset.
\begin{lemma}
  \label{lemma.hereditary_algebra}
  $\setalgebra$ is a subalgebra of $\setalgebra$ if and only if $\sim$
  is hereditary.
\end{lemma}
\begin{proof}
  The if part is the same as for age algebras. For the reciprocal,
  consider a set $C$, its equivalence class $\tau$, and define
  $e_d\in\agealgebra$ as $\sum \chi_{\tau_i}$ where $i$ ranges through
  the finite equivalence classes of sets of size $d$. Equivalently,
  $e_d$ is the (infinite) sum of all subsets of $E$ of size $d$.
  \begin{displaymath}
    \chi_\tau = \sum_{A\suchthat |A|=|C|+d} c_C^A\, A,\,
  \end{displaymath}
  where $c_C^A=|\{X\subset A \suchthat X\sim C\}|$. One recognizes the
  coefficients appearing in the hereditary condition, and it follows
  that  $c_C^A=c_C^B$ whenever $A\sim B$.
\end{proof}

\begin{proposition}
  Let $A$ be a graded subalgebra of $\K[X]$ that admits a basis
  $(B_i)_{i\in I}$ such that each monomial of $\K[X]$ appears in
  exactly one $B_i$. Then, there exist a hereditary equivalence
  relation $\sim$ whose algebra $\agealgebra(\sim)$ is isomorphic to
  $A$.
\end{proposition}
\begin{proof}
  Write $S_i$ the support of $B_i$. By construction, $(S_i)_{i\in I}$
  forms a partition of the monomials of $\K[X]$. Consider the induced
  equivalence relation of $E^{<\omega}$, with equivalence classes
  given by $\tau_i :=\dvect^{-1}(S_i)$ for $i\in I$. Let
  $\agealgebra(\sim)$ be the subspace of $\setalgebra$ spanned by the
  $\chi_{\tau_i}$.

  We now prove that the basis elements $B_i$ of the former are mapped
  one to one by $\phi$ to the basis elements $\chi_{\tau_i}$ of the
  latter.

  Up to rescaling the variables in $X$ once for all, we may assume
  without loss of generality that the basis elements of degree $1$ are
  sums of variables. In particular, $A$ contains $e=x_1+\dots,x_n$ and
  therefore also
  $e^d = \sum_{\dvect\suchthat |\dvect|=d} \frac{x^\dvect}{\dvect!}$.
  Up to rescaling the $B_i$'s, we may therefore assume without loss of
  generality that each $B_i$ is of the form
  $\sum_{\dvect\in S_i} x^\frac{\dvect}{\dvect!}$, where $S_i$ is the
  support of $B_i$. Therefore, $\phi(B_i)=\chi_{\tau_i}$, as desired.

  It follows that $\agealgebra(\sim)$ is indeed an algebra, and
  therefore, using Lemma~\ref{lemma.hereditary_algebra}, that $\sim$
  is hereditary.
\end{proof}

\subsection{Case of orbital algebras}

As pointed out by Cameron~\cite{Cameron.1990}, invariant rings of
finite permutation groups are special cases of orbital algebras. They
are also algebras of relational structures admitting a finite
monomorphic decomposition (see Example~A.16
of~\cite{Pouzet_Thiery.AgeAlgebra1}). The converse holds, in the sense
that, possibly up to some straightforward quotienting, invariant rings
of finite permutation groups are exactly the orbital algebras of
groups admitting a finite monomorphic decomposition.

Let us state this more precisely. Let $G$ be a permutation group on a
set $E$. Choose any relational structure $R$ encoding the orbits of
$G$. Note that the definition of a monomorphic decomposition depends
only on the isomorphism relation between finite subsets of $E$, and
thus is independent of the chosen relational structure. We can thus
forget about the relational structure, and all the concepts of minimal
monomorphic decomposition, monomorphic dimension, etc, are well
defined for $G$ itself.

\begin{theorem}
  \label{theorem.orbit_algebra_finite_monomorphic}
  Let $G$ be a permutation group on a set $E$, and assume that the
  minimal monomorphic decomposition $(E_x)_{x\in X}$ of $E$ is finite.

  Then $G$ induces a finite subgroup $\tilde G$ of the symmetric group
  $\sg_X$ on $X$. If the components are all infinite, then the orbital
  algebra is isomorphic to the invariant ring $\K[X]^{\tilde{G}}$ of
  $\tilde{G}$. Otherwise, it is isomorphic to the quotient thereof
  obtained by setting $x^{|E_x|+1}=0$ for all $x$ such that $E_x$ is
  finite.
\end{theorem}

\begin{proof}
  By Corollary \ref {cor:automorphism},   every
  permutation $\sigma\in G$ induces a permutation $\overline \sigma$ of the
  components $(E_x)_{x\in X}$, which we can identify with a
  permutation of $X$. Choose any relational structure encoding the
  orbits of $G$, and consider the isomorphism $\Phi$ from $K[X]^R$ to
  $\agealgebra(R)$ as in the proof of
  Theorem~\ref{theorem.ageAlgebraPoly}. If all components are
  infinite, it is easy to see that $K[X]^R$ is nothing but the
  invariant ring $K[X]^{\tilde G}$. Otherwise, the same quotienting
  occurs as in~Theorem~\ref{theorem.ageAlgebraPoly}.
\end{proof}

\section{Finite generation}
\label{section.finiteGeneration}

This section is devoted to our main result: the combinatorial
characterization of relational structures admitting a finite
monomorphic decomposition whose algebra is finitely generated. We
setup the ground with a special case and an example before proceeding
to the general case. We conclude with the case of tournaments.

\subsection{Finite generation for bounded profiles}

We recall the following result.
\begin{theorem}[Theorem 1.5 of~\cite{Pouzet_Thiery.AgeAlgebra1}]
  \label{theorem.boundedProfile}
  Let $R$ be a relational structure with $E$ infinite. Then, the
  following properties are equivalent:
  \begin{enumerate}
  \item[(a)] The profile of $R$ is bounded.
  \item[(b)] $R$ is almost-monomorphic.
  \item[(b')] $R$ has a monomorphic decomposition into finitely many
    parts, at most one being infinite.
  \item[(c)] $R$ is almost-chainable.
  \item[(d)] The Hilbert series is of the following form, with
    $P(Z)\in \N[Z]$ and $P(1)\ne 0$:
    \begin{displaymath}
      \hilbert_R=\frac {P(Z)} {1-Z}\ .
    \end{displaymath}
  \item[(e)] The age algebra is a finite dimensional free-module over
    the free-algebra $\K[e]$, where $e:=\sum_{a\in E} \{a\}$; in
    particular it is finitely generated and Cohen-Macaulay.
  \end{enumerate}
\end{theorem}
See ~\cite{Pouzet_Thiery.AgeAlgebra1} for the definition of almost-monomorphy and almost-chainability.
Note that, in~\cite{Pouzet_Thiery.AgeAlgebra1}, Theorem~1.5 is stated
before the introduction of monomorphic decompositions, and thus does
not mention (b'). The equivalence between (b) and (b') follows from
Theorem~2.25 of~\cite{Pouzet_Thiery.AgeAlgebra1}. 
\subsection{Proof of non finite generation on a prototypical example}

The age algebra of a relational structure $R$ admitting a finite
monomorphic decomposition is not necessarily finitely generated.
A prototypical example is this:

\begin{example}
  \label{example.notfinitelygenerated}%
  Let $G$ be the direct sum $K_{(1, \omega)}\oplus \overline K_\omega$
  of an infinite wheel and an infinite independent set. There are two
  infinite monomorphic parts, $E_1$ the set of leaves of the wheel
  and $E_2$ the independent set, and one finite, $E_3$, containing the
  center $c$ of the wheel. Each isomorphism type consists of a wheel
  and an independent set, so the Hilbert series is
  $\hilbert_G(Z)=(1+\frac{Z^2}{1-Z})\frac{1}{1-Z}=\frac{1-Z+Z^2}{(1-Z)^2}$.

  What makes this relational structure special is that the monomorphic
  decomposition $(E_1,E_2,E_3)$ is minimal, whereas $(E_1,E_2)$ is
  \emph{not} a minimal monomorphic decomposition of the restriction of
  $R$ to $E_1\cup E_2$. We now prove that this causes the age algebra  $\agealgebra(G)$
  not to be finitely generated.
  Consider the subalgebra $U:=\K[e]$, where $e:=\sum_{a\in V(G)} \{a\}$. In each degree
  $d$, it is spanned by the sum $b_d$ of all subsets of size $d$ of
  $E$, since $e^d = d!b_d$. Key fact: any element $s$ of
  $\agealgebra(G)$ can be uniquely written as $s =: a(s) + b(s)$ where
  $b(s)$ is in $U$, and all subsets in the support of $a(s)$
  contain the unique element $c$ of $E_3$. Note in particular that
  $a(s)a(s')=0$ for any $s,s'$ homogeneous of positive degree.
  
  Suppose that  $S$ is a finite generating set of $\agealgebra(G)$; we may suppose that $S$ is made  of
  homogeneous elements of positive degree.  By the key observation above,
  $\{a(s), s\in S\}$ generates $\agealgebra(G)$ as a $U$-module.
  It follows that the graded dimension of $\agealgebra(G)$ is bounded by
  $|S|$. But this graded dimension is the profile of $G$ which grows linearly. This gives a contradiction.
\end{example}

\subsection{Combinatorial characterization}

The previous example suggests that the finite generation of the age
algebra is related to the behavior of the minimal monomorphic
decomposition with respect to restriction. This is indeed the case,
and we get a complete characterization of when the age algebra is
finitely generated.

\begin{definition}
  Let $(E_x)_{x\in X}$ be the minimal monomorphic decomposition of
  $R$. Whenever restricting $R$ to some union $\bigcup_{x\in X'} E_x$
  of infinite monomorphic parts, $(E_x)_{x\in X'}$ remains a
  monomorphic decomposition. If it always remains minimal, the
  decomposition $(E_x)_{x\in X}$ is called \emph{hereditary minimal}.
\end{definition}

\begin{remark}
  It is sufficient to check the condition on pairs of infinite parts.
  Namely, a monomorphic decomposition $(E_x)_{x\in X}$ is hereditary
  minimal iff, for any two infinite monomorphic parts $E_x$ and
  $E_y$, the monomorphic decomposition $(E_x,E_y)$ of the restriction
  of $R$ on $E_x\cup E_y$ is minimal.
\end{remark}

\begin{theorem}
  \label{theorem.finiteGeneration}
  Let $R$ be a relational structure admitting a finite monomorphic
  decomposition. Let $(E_x)_{x\in X}$ be its minimal monomorphic
  decomposition, and $X_\infty$ be the set of indices of the infinite
  monomorphic parts. Then, the following propositions are equivalent:
  \begin{enumerate}[{(a)}]
  \item The monomorphic decomposition is hereditary minimal;
  \item The age algebra $\agealgebra(R)$ is finitely generated;
  \item For some large enough integer $D$, the age algebra $\agealgebra(R)$ contains (a
    triangular deformation of) the free subalgebra $\sym(x^{D} ,x\in
    X_\infty)$ and is a module of finite type thereupon.
  \end{enumerate}
\end{theorem}

The implication (c) $\Rightarrow$ (b) is immediate; we prove
separately (a) $\Rightarrow$ (c) and (b) $\Rightarrow$ (a).

For the former implication, the proof is based on the Stanley-Reisner
ring approach of~\cite{Garsia_Stanton.1984} to construct generators of
the invariant rings of a permutation group as a module over symmetric
functions (also dubbed chain-product trick by the second author),
together with the layer addition lemma used
in~\cite{Pouzet_Thiery.AgeAlgebra1} to prove that the Hilbert series
is a rational fraction. For this, we need the tools and notations of
Section~3 of~\cite{Pouzet_Thiery.AgeAlgebra1}, albeit with a small
variant on the chosen total order on the monomials of $\K[X]$; this
variant is necessary to handle the fact that hereditary minimality is
only about the infinite components $E_x$ of the relational structure.

Say for short that a variable $x\in X$ is \emph{finite} if $E_x$ is
finite. Write $\deg_{<\infty}(m)$ for the degree of a monomial $m\in
\K[X]$ in the finite variables. To compare two monomials $m$ and $m'$
in $\K[X]$, first compare their degree in the finite variables;
e.g. set $m>m'$ if $\deg_{<\infty}(m) < \deg_{<\infty}(m')$; in case
of tie, proceed as in~\cite{Pouzet_Thiery.AgeAlgebra1} by comparing
the shapes of $m$ and $m'$, and breaking ties with the usual
lexicographic order.  When there is no finite variable, nothing
changes. Leading monomials, chain support, etc are defined as
in~\cite{Pouzet_Thiery.AgeAlgebra1}. With this order on monomials,
Lemma~3.2 of~\cite{Pouzet_Thiery.AgeAlgebra1} becomes:
\begin{lemma}\label{lemma.addlayer}
  Let $m$ be a leading monomial, and $S\subseteq X$ be a layer of $m$
  with $S\subset X_\infty$. Then, $m x_S$ is again a leading monomial.
\end{lemma}
\begin{proof}
  Proceed as in the proof of Lemma~3.2
  of~\cite{Pouzet_Thiery.AgeAlgebra1}. Since we are considering only
  layers $S\subset X_\infty$, the tweak in the total order on
  monomials does not interfere.
\end{proof}

\begin{proof}[Proof of (a) $\Rightarrow$ (c) of Theorem~\ref{theorem.finiteGeneration}]
  \def\d{d} Consider the restriction $R'$ of $R$ on its infinite
  components. Since the monomorphic decomposition is hereditary
  minimal, the monomorphic decomposition $(E_x)_{x\in X_\infty}$
  remains minimal. Take $\d$ large enough as in Lemma~2.15
  of~\cite{Pouzet_Thiery.AgeAlgebra1}. By Corollary~2.16, for
  $i=1,\dots,k$, the collection $C$ of all sets of size $\d i$ and
  shape $(\d,\dots,\d)$ is closed under orbits. Identifying the
  elementary symmetric function $e_i(x^d, x\in X_\infty)$ in $\K[X]$
  with (a constant factor of) the sum of all the elements of $C$ in
  the set algebra, we derive that this symmetric function belongs to
  the age algebra of $R'$.

  Let $\overline C$ be the closure of $C$ under orbits in $R$, and let
  $\overline e_i\in \agealgebra(R)$ be the sum of the elements in
  $\overline C$. Note that $\overline e_i=e_i(x^\d, x\in X_\infty)$,
  up monomials containing variables $x$ not in $X_\infty$ and thus
  strictly smaller. It follows that $\K[\overline e_i, i=1,\dots,k]$
  is a triangular deformation of the ring of symmetric functions
  $\sym(x^\d, x\in X_\infty)$. In particular, it's the free
  commutative algebra generated by $(\overline e_i)_{i=1,\dots,k}$ in
  the age algebra.

  \noindent\textbf{Claim 1:} Let $m$ be a leading monomial with chain
  support $C$. Assume that $m=x_S^d m'$ for some leading monomial
  $m'$, $S\in C$, and $d\geq0$. Then,
  \begin{displaymath}
    \lm( o(m') \overline e_{|S|} ) = m\,;
  \end{displaymath}
  \begin{proof}
    We say for short that a variable $x\in X$ is \emph{finite} if
    $E_x$ is finite. Write $\deg_{<\infty}(c)$ for the degree of a
    monomial $c$ in the finite variables.

    Take a monomial $ab$ appearing in the product $o(m') \overline
    e_{|S|}$.  If $a=m'$ and $b=x_S^d$, then we recover $m$ and we are
    done.

    Note that $\deg_{<\infty}(a)\geq \deg_{<\infty}(m')$. If the
    comparison is strict or if $\deg_{<\infty}(b)>0$, then
    $\deg_{<\infty}(ab)> \deg_{<\infty}(m)$ and thus $ab<m$ as
    desired.

    Otherwise, $\deg_{<\infty}(a)\geq \deg_{<\infty}(m')$, and $b$ is
    of the form $x_{S'}^\d$ for some $S'\subseteq X^\infty$ with
    $|S'|=|S|$. Given that the shape of $a$ is at most that of $m'$,
    that the shape of $x_{S'}^d$ coincide with that of $x_{S}^d$, and
    that $S$ is in the chain support of $m'$, the shape of $ab$ is at
    most that of $m$. If equality does not holds, we are
    done. Otherwise, the shape of $a$ coincides with that of $m'$, and
    $S'$ is in the chain support of $m'$, and it follows that
    $ab\leq_\lex m$, and therefore $ab\leq m$, as desired.
  \end{proof}

  \def\genlm{G}

  Fix a chain $C:=\emptyset\subset S_1\subset\dots\subset S_r\subseteq
  X$, and let $\lm_C$ be the set of leading monomials of the age
  algebra with this chain support. As in Section~3.2
  of~\cite{Pouzet_Thiery.AgeAlgebra1}, consider
  $\ideal[J]:=\K.\lm_C\oplus \ideal$ of $\K[S_1,\dots,S_l]$; By
  Lemma~\ref{lemma.addlayer}, this is a monomial ideal. Dickson's
  Lemma states that $\ideal[J]$ is finitely generated as an ideal,
  that is as a module over $\K[S_1,\dots,S_l]$. It is in fact also
  finitely generated as a module over $\K[S_1^\d,\dots,S_l^\d]$, with
  a canonical finite set $\genlm_C$ of monomials as
  generators (see e.g. Lemma~2.3.2 (a) of~\cite{Sturmfels.AIT}).

  \def\gen{\mathcal G}

  Let $\gen$ be the finite collection of the orbitsums whose leading
  monomial is in $\genlm_C$ for some chain $C$. We conclude by proving
  the following claim.

  \noindent\textbf{Claim 2:} $\gen$ generates $\agealgebra$ as a module
  over $\K[\overline e_i]$.

  Take an orbitsum $o(m)$ in $\agealgebra$ with $m$ its leading
  monomial, and let $C$ be the chain support of $m$. Assume by
  induction that, for any leading monomial $m'<m$, the orbitsum
  $o(m')$ is in the $\K[\overline e_i]$-module spanned by $\gen$. We
  prove that this holds for $o(m)$ too.

  If $m$ is in $\genlm_C$, $o$ is in $\gen$ and we are done.
  Otherwise, write $m$ as $m=s_S^\d m'$ where $S\subset
  X_\infty$. Using Claim~1, $o(m)$ is in the $\K[\overline
  e_i]$-module spanned by $o(m')$ together with orbitsums of strictly
  smaller leading monomials. Therefore, by induction, $o(m)$ is in the
  $\K[\overline e_i]$-module spanned by $\gen$ as desired.
\end{proof}

\begin{proof}[Proof of (b) $\Rightarrow$ (a) of Theorem~\ref{theorem.finiteGeneration}]
  Assume that the monomorphic decomposition is not hereditary minimal
  but the age algebra $\agealgebra:=\agealgebra(R)$ is finitely generated.
  
  The proof follows the same path as in
  Example~\ref{example.notfinitelygenerated}: we first reduce the
  problem to the monomorphic dimension $2$ case. Then, we construct a
  subalgebra $\mathcal B$ of $\agealgebra$ which is "small" (dimension
  $1$ in each degree) compared to $\agealgebra$ (dimension
  asymptotically equivalent to $d$ in degree $d$).  Using the fact
  that $\agealgebra$ is finitely generated as an algebra we prove that
  $\agealgebra$ is a finitely generated module over $\mathcal B$. This is
  impossible dimension-wise.

  Claim 1: we may assume without loss of generality that $R$ is of
  monomorphic dimension $2$. Otherwise, consider two monomorphic
  parts $E_1$ and $E_2$ such that $R$ restricted to $H=E_1\cup E_2$
  is monomorphic. By the minimality of the decomposition, $H$ is not a
  monomorphic part of $R$: there exist two finite subsets $A$ and $A'$
  of $E$ such that $A$ and $A'$ are not isomorphic yet coincide
  outside $H$. Let $G$ be the finite subset
  $A \backslash H=A'\backslash H$. Then, $R$ restricted to $H\cup G$
  is of monomorphic dimension $2$. The age algebra of the restriction
  is a quotient of the age algebra of $R$ and is therefore still
  finitely generated, as desired.

  Keeping the above notations, we now have $E=H\cup G$ where
  $H=E_1\cup E_2$.
  Consider the graded subalgebra $\mathcal B$ of $\agealgebra$
  spanned in each degree $d$ by the sum $e_d(E)$ of all subsets of
  size $d$ ($\mathcal B$ can be alternatively defined as the free
  graded commutative algebra generated by the sum of all points
  $e_1(E)$, or the age algebra of the trivial relational structure on
  $E$).
  
  Claim~2: $\agealgebra$ is a finite-module over $\mathcal B$. This
  yields the desired contradiction because the graded dimension of
  such a module is bounded by a constant, whereas by
  Lemma~2.15 of~\cite{Pouzet_Thiery.AgeAlgebra1}, the graded dimension of
  $\agealgebra$ is asymptotically equivalent to $d$.
  
  It remains to prove Claim~2.
  
  Since the restriction of $R$ to $H$ is monomorphic, any orbitsum of
  degree $d$ either contains all subsets of size $d$ of $H$, or none.
  Therefore, any $s$ in the age algebra decomposes uniquely as
  $s=:a(s) + b(s)$, where $b(s)$ is in $\mathcal B$, and all the
  subsets in the support of $a(s)$ intersects $G$ non trivially.
  
  Key fact: any product $a(s_1)\cdots a(s_k)$ of $k>|G|$ homogeneous
  elements $s_i$ of positive degree is zero ($k>|G|$ subsets of $E$
  which intersect $G$ non trivially cannot be disjoint).
  
  Let $S$ be a finite generating set of the age algebra made of
  homogeneous elements of positive degree. Then, $\{a(s), s\in S\}$
  generates $\agealgebra$ as a $\mathcal B$-algebra. By the above key fact
  the finite collection of all products $a(s_1)\cdots a(s_k)$ of
  $k\leq|G|$ elements of $S$ generates $\agealgebra$ as a $\mathcal
  B$-module.
\end{proof}

\subsection{Finite generation for tournaments}

The existence of a very simple tournament (Example~A.8
of~\cite{Pouzet_Thiery.AgeAlgebra1}) %
whose age algebra is not finitely generated is not an accident; in
fact, the age algebra of a tournament is very seldom finitely
generated.
\begin{theorem}
  \label{theorem.finiteGenerationTournament}
  The age algebra of a tournament $T$ is finitely generated if and
  only if the profile is bounded.
\end{theorem}

This is a consequence of Theorem~\ref{theorem.finiteGeneration} thanks
to some simple remarks and a structural theorem on the
monomorphic parts of a tournament.
\begin{remarks} \label{remark.tournamentMonomorphicDecomposition}
  \begin{enumerate}[(a)]
  \item A monomorphic part of size at least $4$ is acyclic;
  \item The union of two monomorphic parts of size at least $4$ is
    acyclic;
  \item A minimal monomorphic decomposition of a tournament with
    at least two infinite parts cannot be hereditary minimal.
  \end{enumerate}
\end{remarks}
\begin{theorem}[Boudabbous-Pouzet~\cite{Boudabbous_Pouzet.2009}]
  \label{theorem.tournamentPolynomialGrowth}
  Let $T$ be an infinite tournament whose profile is bounded above by
  a polynomial. Then, $T$ is a lexicographical sum $\sum_{i\in D} A_i$
  of acyclic tournaments $A_i$ indexed by a finite tournament $D$.
\end{theorem}

\begin{proof}[Proof of Theorem~\ref{theorem.finiteGenerationTournament}]
  Suppose that $\profile_T$ is bounded. Then, by
  Theorem~\ref{theorem.boundedProfile}, $\agealgebra(T)$ is finitely
  generated. Conversely, suppose that $\agealgebra(T)$ is finitely
  generated.  Then, the profile $\profile_T$ has polynomial growth,
  and by Theorem~\ref{theorem.tournamentPolynomialGrowth} the minimal
  monomorphic decomposition of $T$
  is finite. Applying Theorem~\ref{theorem.finiteGeneration}, this
  decomposition is hereditary minimal, and by
  Remark~\ref{remark.tournamentMonomorphicDecomposition}~(c), it has a
  single infinite monomorphic maximal part. Therefore, the profile is
  bounded.
\end{proof}

Theorem~\ref{theorem.finiteGenerationTournament} was stated as
Theorem~3.5 in~\cite{Pouzet.2006.SurveyProfile}. The argument provided
there expands on the idea given in
Example~\ref{example.notfinitelygenerated} but, as stated, is
incorrect. The correction is straightforward. Assume that the
tournament $T$ is a lexicographical sum $T=\sum_{i\in D} A_i$ of
acyclic tournaments -- indexed by a finite tournament $D$ -- at least
two of which, $A_j$, $A_k$, are infinite. Then, according to Lemma~9
of~\cite{Boudabbous_Pouzet.2009} $T$ contains a sub-tournament
$T'=\sum_{i\in D'} A'_i$ with the same property where $D'$ has at most
$5$ elements and all $A'_i$ but two, say $A'_{j'}$ and $A'_{k'}$, are
singletons. Supposing $D'$ with minimum size (hence between $3$ and
$5$), the union of the $A'_{i'}$'s, for $i'\not \in \{ j', k'\}$, is
the kernel $\kernel (T')$ of $T'$. As stated in Theorem~3.5, if
$\agealgebra(T)$ is finitely generated, $\agealgebra(T')$ is finitely
generated. Then we may repeat the proof given in
Example~\ref{example.notfinitelygenerated} with $\kernel (T')$ playing
the role of $\{c\}$ and obtain a contradiction (the proof of
Theorem~3.5 deals only with $D$ of size $3$, in which case
$\kernel (T')$ is a singleton). %

\section{Invariant rings of permutation groupoids}
\label{section.invariant.groupoids}

\def\chain{K_\infty}

The common point of invariant rings of permutation groups and other
interesting examples like the rings of quasi-symmetric polynomials
(Example~A.18 of~\cite{Pouzet_Thiery.AgeAlgebra1}) is that they can be
realized as age algebras of a relational structure $R$ of the form
$\chain \wr (X,\rho_i)$, where $X$ is a finite set and $\chain$ is an
infinite clique (or similar monomorphic relational structure). In this section,
we study further such age algebras which we call \emph{invariant rings
  of permutation groupoids}.

Our motivations are twofold. On one hand, relate, in this simpler yet
rich setting, the properties of the profile to algebraic properties of
the invariant ring. In particular, find conditions under which the
invariant ring is Cohen-Macaulay.  On the other hand, generalize the
theory, algorithms, and techniques of invariant rings of permutation
groups to a larger class of subrings of $\K[X]$.

\subsection{Permutation groupoids, inverse monoids and representations}
\label{section.groupoids}

In this section, we briefly review the groupoid and inverse monoid
structures on the local isomorphism of a relational structure, and
their representations. For details, see
e.g.~\cite[Chapter 4]{Lawson.1998.InverseSemigroups}, \cite{Steinberg.2006.Moebius}.

Let $G:=\loc(X,\rho_i)$ be the collection of local isomorphisms of a
finite relational structure $(X, \rho_i)$. Recall that it can be
endowed with a groupoid structure by composing two local isomorphisms
$f$ and $g$ whenever the codomain of the first agrees with the domain
$\dom g$ of the second.

The properties of $G$ are as follows:
\begin{itemize}
\item $G$ contains all local identities of $X$;
\item $G$ is stable under composition, inverse, and restrictions.
\end{itemize}

We call \emph{permutation groupoid} on a set $X$ a collection of local
bijections of $X$ satisfying the above properties, endowed with the
composition product.
\begin{proposition}
  Any permutation groupoid on a finite set $X$ can be realized as
  groupoid of local isomorphisms on some relational structure on $X$.

  A permutation group on $X$ induces a permutation groupoid on $X$ by
  considering all the restrictions of its automorphisms.
\end{proposition}

Alternatively, $G$ can be endowed with an inverse monoid structure,
which we denote by $\overline G$ by taking the partial composition as
product:
\begin{displaymath}
  f g:
  \begin{cases}
    g^{-1}(\im g \cap \dom f) &\bijection f(\im g \cap \dom g)\\
    x \mapsto f(g(x))
  \end{cases}\ .
\end{displaymath}

The groupoid and inverse monoid structures are tightly related through
their algebras. Define as usual the groupoid algebra $\K.G$ of $G$ as
the vector space of formal linear combinations of $G$, endowed with
the product obtained by extending by bilinearity the groupoid product
of $G$.

Define similarly the monoid algebra of $G$ starting from its inverse
monoid structure. The latter algebra is isomorphic to $\K.G$, by
mapping a partial bijection $f$ to
$\overline f = \sum_{A\subset \dom f} f_{\restriction A}$ in $\K.G$,
where $A$ is the domain of $f$. The inverse isomorphism can be defined
by inclusion-exclusion. In the sequel, we identify both algebras,
interpreting $(\overline f)_{f\in G}$ as an alternative basis of
$\K.G$. The map $f \mapsto \overline f$ also defines an embedding of
$\overline G$ in $\K.G$.

Because of this isomorphism, the representations of $G$ and
$\overline G$ coincide. In particular, they are semi-simple.

\subsection{Invariant rings of permutation groupoids}

Take a relational structure $R$ of the form $\chain \wr (X,\rho_i)$.
The running example to keep in mind is that of quasi-symmetric
functions (Example~A.18 of~\cite{Pouzet_Thiery.AgeAlgebra1}). The
monomorphic decomposition $(E_x:=\chain\times \{x\})_{x\in X}$ of $R$
is minimal, and even hereditary minimal. This makes the age algebra
into a subring of $\K[X]$.

We extend the natural action of a permutation of $X$ on $\K[X]$ as
follows. Take a monomial $X^\dvect$ and a local bijection $f$.  If the
\emph{support} of $X^\dvect$ (that is: $\{x\in X \suchthat d_x>0\}$)
coincides with the domain of $f$, set $f.X^\dvect := \prod_{x, d_x>0}
f(x)^{d_x}$; otherwise set $f.X^\dvect:=0$.

\begin{lemma}
  Let $A$ and $B$ be two finite isomorphic subsets of $R$. Then, there
  exists a local isomorphism $f$ of $(X,\rho_i)$ such that $X^B =
  f.X^A$. In particular, the shape of $A$ and $B$ are identical.
\end{lemma}
\begin{proof}
  Let $g$ be a local isomorphism from $A$ to $B$. Two points of $A$ in
  the same monomorphic part must be sent by $g$ to the same
  monomorphic part. Therefore, there exists a local isomorphism $f$
  of $(X,\rho_i)$, and local isomorphisms $(f_x)_{x\in X}$ of $\chain$
  such that every point $(j, x)$ of $A$ is mapped to $g(j,x) =
  (f_x(j), f(x))$.
\end{proof}

This proposition should really be interpreted as follows. The
monomorphic parts are strings, and the finite subsets of $E$ are sets
of beads threaded on those strings and sliding freely on them.
Furthermore some local permutations of the non empty strings are
allowed, corresponding to the local isomorphisms of $(X,\rho_i)$. For
example, for quasi-symmetric functions, one may move the string $1$ to
the string $3$, and the string $2$ to the string $4$, assuming that
all other strings contain no beads.

\begin{corollary}
  The age and therefore the age algebra depend only on the groupoid
  $G$ of the local isomorphisms of $(X, \rho_i)$. Reciprocally, the
  age algebra characterizes the groupoid.
\end{corollary}
By analogy with invariant rings of permutation group, we therefore
write it $\K[X]^G$, and call it \emph{invariant ring of the
  permutation groupoid} $G$. If $G$ is induced by a permutation group,
then both of their invariant rings coincide.

The basic notions of invariant rings of permutation groups, like
\emph{$G$-isomorphic monomials} and \emph{$G$-orbits} extend
straightforwardly; in particular one may define the \emph{orbit sum}
$O(X^D)$ of a monomial $X^D$ as the sum of all the monomials in its
orbit.

\begin{proposition}
  The orbitsums form a vector space basis of the invariant ring
  $\K[X]^G$. The later is a graded connected commutative algebra which
  contains symmetric functions in $X$.
\end{proposition}

\begin{remark}
  The action being by local permutation, the isomorphism of two
  monomials does not depend on the actual values of the exponents, but
  only on the partition of $X$ induced by them.

  Formally: let $f$ be a function from $\N$ to $\N$ such that
  $f(0)=0$, and define $f(X^D):=\prod_x x^{f(D_x)}$. Then, if $X^D$ and
  $X^{D'}$ are isomorphic then so are $f(X^D)$ and $f(X^{D'})$.
\end{remark}

In this setting, Lemma~3.2 of~\cite{Pouzet_Thiery.AgeAlgebra1} becomes
immediate.
\begin{corollary}
  Consider the lexicographic monomial order. If $X^D$ is a leading
  monomial in $\K[X]^G$, and if $S\subseteq X$ is a layer of $X^D$, then
  $X^D X_S$ is again a leading monomial.
\end{corollary}
\begin{proof}
  Duplicating a layer in $X^D$ amounts to apply a strictly increasing
  function $f$. This function preserves both isomorphism and the
  monomial order.
\end{proof}

\subsection{Permutation groupoids versus permutation groups}

Here, we discuss briefly how the action of a permutation groupoid on
polynomials differs from that of a permutation group.

\subsubsection{Restriction}

The \emph{restriction} $G_{\restriction X'}$ of a permutation groupoid
$G$ to a subset $X'$ is the set of all local functions $f$ in $G$ such
that $\dom f\subseteq X'$ and $\im f \subseteq X'$, which is again a
permutation groupoid.  Furthermore, the orbits of
monomials in $\K[X']$ are unchanged by this restriction. In particular, the
invariant ring of $G_{\restriction X'}$ is simply the quotient of the
invariant ring of $G$ obtained by killing all the variables $x_i$ with
$i\not\in X'$.
This simple fact is one of the motivations for considering permutation
groupoids instead of just permutation groups (for which the
restriction to a subset is not clearly defined). This may indeed give
opportunities for induction techniques on the size of the underlying
set.
\begin{proposition}
  Any permutation groupoid on a finite set is the restriction of a permutation
  groupoid induced by a permutation group of some superset. However,
  this superset may need to be infinite.
\end{proposition}
\begin{proof}
  This is an immediate consequence of the fact that every finite
  relational structure $R$ embeds into a countable homogeneous
  structure (structure for which local isomorphisms of finite domain
  extend to automorphisms)~\cite{Fraisse.1954-2}. In several
  instances, the permutation groupoid may be chosen
  finite~\cite{Herwig_Lascar.2000}.
\end{proof}

\begin{examples}
  \label{examples.restrictions}
  (a) The permutation groupoid on $\{1,2,3\}$ generated by the
  rank $1$ local bijection $1\mapsto 2$ is the restriction of the
  permutation group on $\{1,2,3,4\}$ generated by the permutation
  $(1,2)(3,4)$.

  (b) The local automorphism permutation groupoid of the chain $a<b$
  is the restriction of the cyclic group $C_3$ on $\{a,b,c\}$.
  
  (c) Consider a relational structure $R$ such that there exists three
  elements $a,b,c$ and a binary relation $<$ which restricts on
  $\{a,b,c\}$ to the chain $a<b<c$. Typically, $R$ is a chain of
  length at least $3$ (giving $\qsym(X)$ as invariants) or a poset of
  height at least $3$. Then, there exists no relational structure
  $\overline R$ on a finite superset where all local isomorphisms
  extend to global isomorphisms.
  
  Consider indeed the local isomorphism $a\mapsto a, c\mapsto b$, and
  extend it to a global isomorphism $\sigma$ of $\overline R$. It is
  easy to check that $a<\sigma(b)<b<c$ is again a chain, which implies
  that $\sigma(b)\not\in \{a,b,c\}$. By induction,
  $a<\sigma^k(b)<\dots<\sigma(b)<c$ is again a chain, which proves
  that all $\sigma^k(b)$ are distinct. Hence $R'$ is infinite.
\end{examples}

\subsubsection{Multiplicativity}

As for a permutation group, the groupoid algebra $\K.G$ is semi-simple,
and the action of $G$ extends to a linear representation of $G$. However,
for a permutation groupoid, the action is not multiplicative on
polynomials. Take for example $f:=\id_{\{1,2\}}$, $P:=x_1$ and
$Q:=x_2$. This requires a bit of care in the upcoming generalization
of the Reynolds operator and explains why it is not any more a
$\sym$-module morphism.

Multiplicativity can be partially recovered by considering the inverse
monoid $\overline G$, whose on polynomials is given by
$\overline f.X^D := \prod_{x, D_x>0} f(x)^{D_x}$ if
$\{x, D_x>0\} \subseteq \dom f$ and $f.X^D:=0$ otherwise. This action
is multiplicative.

\subsection{The Reynolds operator}

The first essential feature of invariant rings is the so-called
\emph{Reynolds operator} $\reynolds$, a projector on the invariant ring.
The following proposition states that this operator still exists for
invariants of permutation groupoids, albeit missing the important
property of being a $\K[X]^G$-module morphism. In particular, although
$\K[X]^G$ still contains the ring of symmetric polynomials $\sym(X)$,
$R$ is not anymore a $\sym(X)$-module morphism.
Recall also that, for a permutation group, $\K[X]^G$ is the isotopic
component for the trivial representation of $\K.G$ in $\K[X]$;
furthermore, the Reynolds operator is the unique central idempotent of
$\K.G$ projecting on the trivial representation. These properties
fail for a permutation groupoid; in fact $\K[X]$ is not even stable
under the action by $\K.G$.

\begin{proposition}\label{proposition.groupoids.reynolds}
  There exists an idempotent $\reynolds$ in the groupoid algebra $\K.G$ which projects $\K[X]$
  onto the invariant ring $\K[X]^G$:
  \begin{displaymath}
    \reynolds:=\sum_{A\subseteq X} \frac{1}{|\{g\in G\suchthat \dom g = A\}|} \sum_{g\in G\suchthat \dom g = A} g\ .
  \end{displaymath}

  Furthermore, the four following conditions are equivalent:
  \begin{enumerate}[(i)]
  \item $G$ is induced by a permutation group.
  \item $\reynolds$ is a $\K[X]^G$-module morphism;
  \item $\reynolds$ is a $\sym(X)$-module morphism;
  \item $\ker \reynolds$ is a $\sym$-module;
  \end{enumerate}
\end{proposition}

\begin{proof}
  By construction, $\reynolds$ is in the groupoid algebra. One easily
  checks that, up to a scalar factor, the image of a monomial by
  $\reynolds$ is the orbitsum of that monomial, and the image of an
  orbitsum is itself. Therefore, it projects onto the invariant ring.

  (i) $\Rightarrow$ (ii), (ii) $\Rightarrow$ (iii), (iii)
  $\Rightarrow$ (iv) are well known or obvious.

  (iv) $\Rightarrow$ (i): Assume that $\ker R$ is a $\sym$-module and
  that $G$ does not come from a permutation group, and let $f:A\mapsto
  B$ be a local bijection which does not extend to a permutation of
  $X$. Let $x_X$ be the product of the variables, and $m$ be a
  monomial with support $A$ with all exponents distinct. Then, $f.m$
  is in the orbit of $m$, whereas $x_X f.m$ is not in the orbit of
  $x_Xm$. Therefore, $\reynolds(m-f.m)=0$ whereas $x_X\reynolds(m-f.m)
  = \reynolds(x_Xm - x_Xf.m)\ne 0$, a contradiction.
\end{proof}

\subsection{Fine grading, the chain product, and degree bounds}

Recall that the \emph{degree bound} $\beta(A)$ of a finitely generated
graded algebra $A$ is the smallest integer such that $A$ is generated
by its elements of degree at most $\beta(A)$.
In~\cite{Garsia_Stanton.1984}, Garsia and Stanton construct an
associated finely-graded algebra of the invariant ring $\K[X]^G$ of a permutation
group $G$ by letting $G$ act on the Stanley-Reisner ring of the
boolean lattice. They use it to exhibit $\sym(X)$-module generators in some
special cases and prove, in general, the quadratic degree bound
$\beta(\K[X]^G)\leq\binom{|X|}2$. A notable feature is that this
approach is combinatorial and thus characteristic free!

In this section, we show that this construction generalizes
essentially straightforwardly to permutation groupoids, and derive
some properties of the invariant ring: degree bounds and finite
generation; in the following section we further derive a necessary
conditions for being Cohen-Macaulay.

As in~\cite{Thiery.AIG.2000}, we follow the basic approach of
realizing the Stanley-Reisner ring as $\K[X]$ endowed with another
product $\star$, called the \emph{chain product}. The chain product
preserves a finer grading, and many algebraic properties of the
invariant ring w.r.t. the chain product transfer back to the usual
product. We refer to \cite{Eisenbud.CA} for background on filtrations
and associated graded algebras.

Given a subset $S$ of $X$, set $x_S:=\prod_{i\in S} x_i$. By
square-free decomposition, any monomial $x^d$ can be identified
uniquely with a \emph{multichain} $S_1\subseteq \dots \subseteq S_k$
of nested subsets of $X$, so that:
\begin{displaymath}
  x^d = x_{S_1} \cdots x_{S_k}\ .
\end{displaymath}
We call each $S_k$ a \emph{layer} of $x$.

The \emph{fine degree} of the monomial $x^d$ is the integer vector
$(r_1,\dots,r_n)$ where each $r_i$ counts the (possibly null) number
of repetitions of the layer of size $i$ of $x^d$. Orbit sums are
homogeneous w.r.t. the fine degree; the invariant ring is therefore
graded as a vector space.

One may compute the fine Hilbert series of the invariant ring as
follows: take a representative $Y$ for each orbit of subsets of $X$;
consider the group $G_Y$ of local automorphisms of domain $Y$; use
Pólya enumeration to compute the fine generating series for monomials
with full support $Y$. Sum all the results.

The fine grading is not preserved by multiplication; however it still
defines a filtration on $\K[X]$.
The \emph{chain product} $\star$ of two monomials $x^d := x_{S_1}
\cdots x_{S_k}$ and $x^{d'} := x_{S_1'} \cdots x_{S_k'}$ is defined
by:
\begin{displaymath}
  x^d \star x^{d'} :=
  \begin{cases}
    x^d x^{d'} &
    \text{ if $\{S_1, \dots, S_k, S'_1,\dots,S'_k\}$ is again a
      multichain of subsets,}\\
    0 & \text{otherwise.}
  \end{cases}
\end{displaymath}
For example, $x_1 \star x_1 = x_1^2$, $x_1 \star x_2=0$, $x_1 x_3^2 \star
x_1x_2x_3^2= x_1^2x_2x_3^4$, and $x_1 x_3^2 \star x_1x_2=0$.

The chain product endows $\K[X]$ with a second algebra structure $(\K[X],\star)$, isomorphic to the quotient
\begin{displaymath}
  \K[x_S, S\subseteq X] / \{x_Sx_{S'}=0\suchthat S\not\subseteq S' \text{ and } S'\not\subseteq S\}\,.
\end{displaymath}
$(\K[X],\star)$ is
also finely graded, fine degrees being added term-by-term.  In fact,
$(\K[X],\star)$ is exactly the associated graded algebra of $\K[X]$ w.r.t.
the fine degree filtration. Beware however that $(\K[X],\star)$
is not an integral domain.

The elementary symmetric functions
\begin{displaymath}
  e_d := \sum_{S\subseteq X, |S|=d} x_S
\end{displaymath}
are still algebraically independent and generate $(\sym(X)_n,\star)$.
This is no longer true for, say, the symmetric powersums. The
following simple fact turns out to be an essential key:
\begin{remark}
  Consider the chain product of a monomial $x_{S_1} \cdots x_{S_k}$ by
  the elementary symmetric function $e_d$. It is the sum of all
  monomials $x_{S_1} \cdots x_S \cdots x_{S_k}$, where $S$ is of size
  $k$, and fits in the multichain $S_1\subseteq\cdots \subseteq
  S\subseteq\cdots\subseteq S_k$. In particular, if $x_{S_1} \cdots
  x_{S_k}$ contains a layer $S$ of size $k$, then $x_{S_1} \cdots
  x_{S_k} \star e_k$ is the unique monomial obtained by replicating
  this layer.
\end{remark}

More generally, $(\K[X]^G,\star)$ is a subring of $(\K[X],\star)$. In
particular, $(\K[X]^G,\star)$ is a $\sym(X)$-module. Furthermore, we may
transfer the following algebraic properties from $(\K[X],\star)$ to
$\K[X]^G$, as in the case of permutation
groups~\cite{Garsia_Stanton.1984}.
\begin{proposition}
  \label{proposition.transfer}
  \begin{itemize}
  \item[(a)] A family $F$ of finely homogeneous invariants of positive
    degree which generates $(\K[X]^G,\star)$, also generates $\K[X]^G$;
  \item[(b)] $\beta(\K[X]^G,\star) \geq \beta(\K[X]^G)$;
  \item[(c)] A family $F$ of finely homogeneous invariants which
    generates $(\K[X]^G,\star)$ as a $\sym(X)$-module also generates
    $\K[X]^G$ as a $\sym(X)$-module;
  \item[(d)] If $(\K[X]^G,\star)$ is a free $\sym(X)$-module, then so is
    $\K[X]^G$.
  \end{itemize}
\end{proposition}
\begin{proof}
  This is a standard fact about filtrations and associated graded
  connected algebras. The key of the proof is that, if $p$ and $q$ are
  finely homogeneous, the maximal finely homogeneous component of $pq$
  is exactly $p\star q$. (a) and (c) follow by induction over the fine
  grading.
  Then, (b) follows straightaway from (a), and (d) from (c) by a
  simple Hilbert series argument.
\end{proof}

The converse of (a) and (b) do not hold. In fact, with most
permutation groups, the degree bound $\beta(\K[X]^G,\star)$ is much larger
than $\beta(\K[X]^G)$. We conjecture that the converse of (c) and (d)
hold.  However (d) does not hold anymore in a slightly larger setting
which includes the $r$-quasi-symmetric polynomials of
F.~Hivert~\cite{Hivert.RQSym.2004}, a counter example being
$\qsym^2(X_3)$ (there is an obstruction in the fine Hilbert series).

We now generalize the quadratic degree bound of Garsia and Stanton to
permutation groupoids.

\begin{theorem}\label{theorem.finiteGenerationPermutationGroupoid}
  Let $G$ be a permutation groupoid acting on $X$ and $n=|X|$. Then,
  the invariant ring $\K[X]^G$ is a finitely generated algebra and
  $\sym(X)$-module, in degree at most $\frac{|X|(|X|+1)}2$. This
  degree bound is tight.
\end{theorem}
Note that, as usual, when $G$ does not act transitively on the
variables, the degree bound can be greatly improved by considering the
symmetric polynomials on each transitive component instead.
\begin{proof}
  The set of orbit sums $o(x_{S_1} \cdots x_{S_k})$, where
  $S_1\subsetneq \cdots \subsetneq S_k$ is a chain, generates
  $(\K[X]^G,\star)$ as a $(\sym,\star)$-module. This transfers back to
  $\K[X]^G$ and $\sym$.
  
  Note that we may need to consider chains with $S_k=X$; hence the
  degree bound of $\frac{|X|(|X|+1)}{2}$ instead of $\binom{|X|}2$ for
  permutation groups. For an example where the bound is achieved,
  consider the group $G$ made of the identity together with all the
  local bijections of $X=\{1,\dots,n\}$ whose domain is of size at
  most $|X|-1$; then, $\K[X]^G$ is freely generated as a $\sym$-module
  by $1$ and the \emph{staircase} monomials $x_1^{d_1}\cdots
  x_n^{d_n}$ with $1 \leq d_i \leq i$.
\end{proof}

\subsection{The Cohen-Macaulay property}

Invariant rings of permutation groups are always Cohen-Macaulay, and
in fact free $\sym(X)$-modules. The key ingredients are that $\K[X]$
is a free $\sym(X)$-module and the Reynolds operator a
$\sym(X)$-module morphism. A more involved result is that, for all
$n$, $\qsym(X_n)$ is also a free
$\sym(X_n)$-module~\cite{Garsia_Wallach.2003}.

As Example~A.19 of~\cite{Pouzet_Thiery.AgeAlgebra1}
shows, this property does not
hold for all permutation groupoids $G$. Still, $\K[X]^G$ and
$(\K[X]^G,\star)$ being finitely generated over $\sym(X)$, they are
Cohen-Macaulay if and only if they are free $\sym(X)$-modules.

\begin{problem}\label{problem.CM}
  Characterize the permutation groupoids $G$ whose invariant ring
  $\K[X]^G$ (or $(\K[X]^G,\star)$) is Cohen-Macaulay.
\end{problem}

In practice, a first test is to compute the (fine) degrees of
tentative free generators of $\K[X]^G$ by dividing the (fine) Hilbert
series of $\K[X]^G$ by that of $\sym(X)$.

The following theorem is a straightforward extension of Theorem~6.1
of~\cite{Garsia_Stanton.1984} applied to the quotient of the Stanley
Reiner ring of the boolean lattice by a permutation group.
\begin{theorem}
  $(\K[X]^G,\star)$ is a free $\sym(X)$-module if and only if the
  incidence matrix between generators and orbits of maximal chains is
  invertible. In particular, for a set $F$ of finely homogeneous
  invariants whose fine degrees are given by the Hilbert series of
  $\K[X]^G$, the three following conditions are equivalent: $F$ spans
  $\K[X]^G$ as a $\sym(X)$-module, $F$ is a free $\sym(X)$-family, and
  $F$ is a $\sym(X)$-basis of $\K[X]^G$.
\end{theorem}
This immediately gives us a necessary condition on the number of
generators.
\begin{corollary}
  If $(K[X]^G,\star)$ is a free $\sym(X)$-module, then it is of rank
  $\frac{|X|!}{|G(X,X)|}$, where $G(X,X)$ is the underlying
  permutation group of $G$.
\end{corollary}
As a consequence, we recover that the ring $\qsym(X)$ of quasi
symmetric polynomials in $X$ has to be of rank $|X|!$ over $\sym(X)$.

\subsection{SAGBI bases}

SAGBI bases (Subalgebra Analog of a Gröbner Bases for Ideals) were
introduced in~\cite{Kapur_Madlener.1989,Robbiano_Sweedler.1990} to
develop an elimination theory in subalgebras of polynomial rings.
Unlike Gröbner bases, not all subalgebras have a finite SAGBI basis,
and it remains a long open problem to characterize those subalgebras
which have a one.
The following theorem states that, as in the case of permutation
groups, invariant rings of permutation groupoids seldom have finite
SAGBI bases. The proof follows the short proof given by the second
author in~\cite{Thiery_Thomasse.SAGBI.2002} for permutation groups,
with some adaptations.  For example $\qsym(X_n)$, represented as a
subring of $\K[X]$, has no finite SAGBI basis whenever $n>1$.  In
particular, $\qsym(X_2)$ becomes the smallest example of finitely
generated algebra which has no finite SAGBI basis (the standard
example being the invariant ring of the alternating group $A_3$).
Still, SAGBI bases and SAGBI-Gröbner bases provide a useful device in
the computational study of invariant rings of permutation
groups~\cite{Thiery.CMGS.2001}, and should play the same role
with permutation groupoids.

\begin{theorem}\label{theorem.SAGBI}
  Let $G$ be a permutation groupoid, and $<$ be any admissible term
  order on $\K[X]$. Then, the invariant ring $\K[X]^G$ has a finite
  SAGBI basis w.r.t. $<$ if, and only if, $G$ is induced by a permutation
  group generated by reflections (that is transpositions).
\end{theorem}
The following proof is a close variant on the short proof given by the
second author in~\cite{Thiery_Thomasse.SAGBI.2002} in the special case
of permutation groups. For the sake of readability and completeness,
we include it in full here. The key fact is that a submonoid $M$ of
$\N^n$ is finitely generated if, and only if, the convex cone
$C:=\R_+M$ it spans in $\R_+^n$ is finitely generated (that is $C$ is
a \emph{polyhedral cone}). For details, see for
example~\cite[Corollary~2.10]{Bruns_Gubeladze.PolytopesRingsKTheory.2009}. In
particular $C$ must be the intersection of finitely many half spaces,
and thus closed for the euclidean topology.
\begin{proof}
  The if-part is easy, a finite SAGBI basis being given by the
  elementary symmetric polynomials in the variables in each
  $G$-transitive components.
  
  Without loss of generality, we may assume $X=\{1,\dots,n\}$ with
  $x_1>\dots>x_n$. Let $M$ be the monoid of initial monomials in
  $\K[X]^G$, seen as a submonoid of $\N^n$, and $C:=\R_+M$ be the
  convex cone it spans in $\R_+^n$.
  
  At this stage, we cannot give an explicit description of $C$, but we
  can construct a convex cone $C'$ which approximates it closely
  enough for our purposes.  By the standard characterization of
  admissible term orders on $\K[X]$, there exists a family of $n$
  linear forms $l=(l_1,\dots,l_n)$ such that $x^d > x^{d'}$ if and
  only if $l(d) >_\lex l(d')$, where we denote by $l(d)$ the $n$-uple
  $l_1(d_1,\dots,d_n),\dots, l_n(d_1,\dots,d_n)$. Given two vectors
  $v$ and $v'$ in $\R_+^n$, we write $v>v'$ if $l(v)>_\lex l(v')$.
  The partial action of $G$ on monomials extends naturally to a
  partial action on $\R_+^n$: whenever the support of
  $v=(v_1,\dots,v_n)$ in contained in the domain of a local bijection
  $f\in G$, $f.v$ is the vector obtained by permuting the non zero
  entries of $v$ according to $f$.  Let $C'$ be the subset of all
  vectors $v$ of $\R_+^n$ such that $v > f.v$ for all $f.v$ in the
  $G$-orbit of $v$. In fact, $C'$ is a convex cone with non empty
  interior (it contains the $n$ linearly independent vectors
  $(1,0,\dots,0), (1,1,0,\dots,0), \dots, (1,\dots,1)$). By
  construction, $M$ consists of the points of $C'$ with integer
  coordinates. It follows that $C\subseteq C'\subseteq \overline C$, where
  $\overline C$ is the topological closure of $C$.
  
  Assume now that $M$ is finitely generated. Then, $C$ is a closed
  convex cone, and $C$ and $C'$ simply coincide.
  
  Assume further that $G$ is not generated by transpositions. Then,
  there exists $a<b$ such that the transposition $(a,b)$ is not in
  $G$, while $a$ is in the $G$-orbit of $b$. Choose such a pair $a<b$
  with $b$ minimal. We claim that there is no transposition $(a',b)$
  in $G$ with $a'<b$.  Otherwise, $a$ and $a'$ are in the same
  $G$-orbit, and by minimality of $b$, $(a,a')\in G$; thus,
  $(a,b)=(a,a')(a',b)(a,a')\in G$. Pick $g\in G$ such that $g.b=a$,
  and for $t\geq 0$, define the vector in $\R_+^n$:
  \begin{displaymath}
    u_t :=
    \left(nt,\ (n-1)t,\ \dots,\ (n-b+2)t,\ n-b+1,\ (n-b)t,
      \ \dots,\ t,\ 1\right)\ .
  \end{displaymath}
  Note that $u_1=(n,\dots,1)$ is in $C$, whereas
  $u_0=(0,\dots,0,n-b+1,0,\dots0)$ is not in $C$ because $g.u_0>u_0$.
  
  Take $t$ such that $0<t\leq1$. Then, the vector $u_t$ has no zero
  coefficients, and in particular its $G$-orbit coincides with its
  orbit w.r.t. the underlying permutation group $G(X,X)$.
  Furthermore, the entries of $u_t$ are all distinct, except when
  $t=\frac{n-b}{n-a'}$ for some $a'<b$, in which case the $a'$-th and
  $b$-th entries are equal. Since $(a',b)\not\in G$, the orbit of
  $u_t$ is of size $|G(X,X)|$, and there exists a unique permutation
  $f_t\in G(X,X)$ such that $f_t.u_t$ is in $C$.
  
  Let $t_0=\inf\{t\geq 0, u_t \in C\}$. If $u_{t_0}\not\in C$, then
  $u_{t_0}$ is in the closure of $C$, but not in $C$, a contradiction.
  Otherwise, $u_{t_0}\in C$, and $t_0>0$ because $u_0\not\in C$.  For
  any permutation $f$, $\{f.u_t, t\geq 0\}$ is a half-line; so, $C$
  being convex and closed, $I_f:=\{t, f.u_t\in C\}$ is a closed
  interval $[x_f, y_f]$. For example, $I_\id=[t_0,1]\subsetneq [0,1]$.
  Since the interval $[0,1]$ is the union of all the $I_f$, there
  exists $f\ne\id$ such that $t_0\in I_f$. This contradicts the
  uniqueness of $f_{t_0}$.
\end{proof}

\subsection{Stability by derivation}

We denote by $\partial_i$ the derivative w.r.t. the variable $x_i$,
and consider the derivation $D := \sum_{i\in X} \partial_i$ on $\K[X]$.
\begin{proposition}\label{proposition.groupoid.derivation}
  Let $G$ be a permutation groupoid. Then, $\K[X]^G$ is stable by the
  derivation $D$ if and only if $G$ is induced by a permutation group.
  On the other hand, $\K[X]^G$ is always stable w.r.t. the action of
  the \emph{rational Steenrod operators} $S_k := \sum_i x_i^{k+1}
  \partial_i$ for $k\geq 0$ (see~\cite{Hivert_Thiery.SA.2002} for
  details on the rational Steenrod operators).
\end{proposition}
\begin{proof}
  The if-part is trivial, since $D$ commutes with the action of the
  symmetric group $\sg_X$ on $\K[X]$. Similarly, the rational Steenrod
  operators always stabilize $\K[X]^G$ because they commute with the
  action of any local bijection on $\K[X]^G$.
  
  Assume now that $\K[X]^G$ is stable by derivation. Let $f:A\mapsto
  B$ be a local bijection such that $A\subsetneq X$, and take $i$ in
  $X\backslash A$. We just need to prove that $f$ extends to a local
  bijection $g$ in $G$ with domain $A\cup \{i\}$.  Applying induction,
  any local bijection in $G$ will then extend to a permutation, as
  desired.
  
  Take a monomial $m$ whose support is $A$ and whose exponents are all
  distinct and at least $2$, and consider the derivation
  $p=D(o(mx_i))$ of the orbitsum of the monomial $mx_i$ in $\K[X]^G$.
  The monomial $m$ occurs in $p$; hence, by invariance of $p$, $f(m)$
  also occurs in $p$, as the derivative of some monomial $g(m x_i)$ in
  the orbit of $mx_i$. By the choice of the exponents of $m$, $f$ and
  $g$ must coincide on $A$, while at the same time $i$ belongs to the
  domain of $g$.
\end{proof}
\begin{example}
  $\qsym(X_2)$ has no graded derivation of degree $-1$.
\end{example}

\section{Hopf algebra structure}
\label{section.hopf}

Let $R$ be a relational structure on a set $E$. In this section, we
propose one approach to try to endow its age algebra with a Hopf
algebra structure by looking at copies of $R$ within $R$.

Assume that there exists two disjoint subsets $E_1$ and $E_2$ of $E$
such that:
\begin{itemize}
\item[(a)] $R$ restricted to $E_i, i=1,2$ is isomorphic to $R$, or
  at least has the same age as $R$.
\item[(b)] The isomorphism type of a set $A\subseteq E_1\cup E_2$ is
  entirely determined by the isomorphism types of $A\cap E_1$ and
  $A\cap E_2$.
\end{itemize}
Define the following graded algebra morphism on the set algebra:
\begin{equation}
  \Delta_{E}^{E_1,E_2}:
  \begin{cases}
    \setalgebra[E] &\twoheadrightarrow \setalgebra[E_1] \otimes \setalgebra[E_2]\\
    A &\mapsto
    \begin{cases}
      A\cap E_1 \otimes A\cap E_2 & \text{if $A\subseteq E_1\cup E_2$,}\\
      0                           & \text{otherwise.}
    \end{cases}
  \end{cases}
\end{equation}

\begin{lemma}
  \label{lemma.hopf}
  $\Delta_E^{E_1,E_2}$ induces a coproduct $\Delta$ on
  $\agealgebra(R)$, that is a graded algebra morphism from
  $\agealgebra(R)$ to $\agealgebra(R)\otimes \agealgebra(R)$.
\end{lemma}

\begin{proof}
  All we have to check is that $\Delta_E^{E_1,E_2}$ indeed maps
  $\agealgebra(E)$ to $\agealgebra(E_1)\otimes \agealgebra(E_2)$.
  Things are easier in the graded dual which is the quotient of the
  set algebra by the isomorphism equivalence relation. There, we need
  to check that the dual product is well defined; that is,
  given two types $\tau_1\in \age(E_1)$ and $\tau_2\in \age(E_2)$,
  the type of the product $A_1\cup A_2$ of two
  representatives $A_1$ and $A_2$ of the types $\tau_1$ and $\tau_2$
  respectively shall be independent of that choice of
  representatives. This is exactly condition (b). By
  condition (a), $\Delta$ can then be interpreted as going from
  $\agealgebra(R)$ to $\agealgebra(R)\otimes\agealgebra(R)$ and is, by
  construction, an algebra morphism.
\end{proof}

Let in addition $\phi_1,\phi_2$ be isomorphisms from $E$ to $E_1, E_2$
respectively, and denote by $E_{i,j}:=\phi_i(E_j)$ the induced copy of $E_j$ inside
$E_i$. Consider the induced bijection
\begin{displaymath}
  \Phi:
  \begin{cases}
    E_{1,1}\times E_{1,2} \times E_2 &\hookrightarrow\!\!\!\!\!\rightarrow E_1\times E_{2,1} \times E_{2,2}\\
    (a,b,c) &\mapsto (\phi_1^{-1}(a), \phi_1\circ\phi_2\circ\phi_1^{-1}\circ\phi_2^{-1}(b),\phi_2(c))
  \end{cases}\,,
\end{displaymath}
and assume further
\begin{itemize}
\item[(c)] $\Phi$ is an isomorphism: namely, if $A,B,C$ are three
  subsets of $E_{1,1}, E_{1,2},E_2$ respectively and
  $A',B',C'=\Phi(A,B,C)$, then $A\cup B\cup C\subset E$ and
  $A'\cup B'\cup C'\subset E$ have the same type in $\age(R)$.
\end{itemize}

\begin{proposition}
  \label{proposition.hopf}
  Let $R$, $E_1$, $E_2$, $\Delta$, $\phi_1$, $\phi_2$, as above
  satisfying (a) and (b) and (c).
  Then the coproduct is coassociative, turning $\agealgebra(R)$ into a
  freely generated graded connected commutative Hopf algebra.  In
  particular, its Hilbert series is of the form:
  \begin{equation}
    \hilbert_R(Z) = \prod_i \frac1{(1-Z^{d_i})}\,,
  \end{equation}
  where the product may be infinite.
\end{proposition}
\begin{proof}
  Working in the graded dual as in the proof of
  Lemma~\ref{lemma.hopf}, we need to check the dual product is
  associative. Take $A,B,C$ representatives of three orbits; without
  loss of generality they may be chosen in $E_{1,1}, E_{1,2}, E_2$
  respectively. Take $A',B',C'=\Phi(A,B,C)$, and note that $A',B',C'$
  are representatives of the same orbit. Condition (c) guarantees
  exactly that the product $(A \cup B) \cup C$ and
  $A'\cup (B'\cup C')$ are in the same orbit, as desired.

  Since $\agealgebra(R)$ is graded connected, most axioms of Hopf
  algebras (in particular concerning the antipode) are satisfied for
  free. The age algebra is then a graded commutative Hopf algebra.
  Using the classical Milnor-Moore theorem, it must be a free
  commutative algebra. The form of the Hilbert series follows.
\end{proof}

\appendix
\section{More examples of relational structures and age algebras}

\subsection{Examples with or without a Hopf age algebra structure}

\begin{example}
  Take the set $(-1,1)$ equipped with either its natural antichain or
  chain structure, and some relational structure $S$ on some set $F$.
  Construct the relational structure $R:=\Q\wr S$ on $E:=\Q\times F$
  by substituting each rational number by a copy of $F$. Define
  $E_1=(-1,0)\times F$ and $E_2=(0,1) \times F$. Consider the affine
  isomorphisms $\phi_1$ and $\phi_2$ mapping $E$ on $E_1$ and $E_2$
  respectively, so that, e.g., $E_{2,1}=(0,\frac12)$.

  Then, the axioms (a), (b), (c) of Proposition~\ref{proposition.hopf}
  are satisfied, and $\agealgebra(R)$ is a Hopf algebra.
  When $S$ is an infinite chain and $\Q$ is equipped respectively with
  its natural antichain or chain structure, one get respectively the
  Hopf algebras of symmetric and quasi-symmetric functions on the
  monomial basis.

  Example~A.9 of~\cite{Pouzet_Thiery.AgeAlgebra1} is obtained by
  taking a finite graph $G$ for $S$ and the antichain on $\Q$. We
  recover the free algebra generated by connected finite restrictions
  of $S$. This generalizes immediately for any relational structure
  $S$.
\end{example}

In the examples above, the construction mimics the standard trick of
doubling the alphabet to construct Hopf algebras (see
e.g.~\cite[Section, 3.2]{Duchamp_Hivert_Thibon.2002}
or~\cite{Hivert.2005,Priez.2013}) which is our original inspiration. The
existence of a coassociative coproduct is a very strong constraint,
and it is not clear whether there exist interesting coassociative
examples where the construction really goes beyond this trick.

\begin{example}
  Let $X:={x_1,\dots,x_n}$, and consider the polynomial rings $\K[X]$
  realized as invariant ring of the trivial permutation groupoid
  $\K[X]=\K[X]^\id$. One can take as relational structure
  $R:=X\times \Q$ where each piece $\{x_i\}\times \Q$ is colored
  differently by a unary predicate. Taking $E_1:=X\times (-1,0)$ and
  $E_2:=X\times (0,1)$, Proposition~\ref{proposition.hopf} endows
  $\K[X]$ with its usual Hopf algebra structure where the generators
  $x_i$ are primitive.
\end{example}

\begin{example}
  The age algebra of a direct sum $R:=K_\omega\oplus\cdots\oplus
  K_\omega$ of $k$ infinite cliques is the algebra of symmetric
  polynomials on $k$ variables (see~\cite[Example
  A.2]{Pouzet_Thiery.AgeAlgebra1}). Since it is a free algebra, it can
  be endowed with a Hopf algebra structure (for example by making its
  generators group-like). Yet, we have not found a way to achieve this
  using Proposition~\ref{proposition.hopf} on this particular
  relational structure $R$.
\end{example}

\begin{example}[The Planar Shuffle Algebra]
  \newcommand{\ta}{((\leaf,\leaf),\leaf)}
  \newcommand{\tb}{(\leaf,(\leaf,\leaf))}
  In~\cite[Section A.4]{Pouzet_Thiery.AgeAlgebra1}, we realized the
  Planar Shuffle Algebra of
  Gerritzen~\cite{Drensky_Gerritzen.2004,Gerritzen.2.2004,Gerritzen.2004}
  as an age algebra.

  Consider the infinite tree $T$ depicted in Figure 1 of~\cite[Section
  A.4]{Pouzet_Thiery.AgeAlgebra1}. Recall that the relational
  structure consists of the infix total order and three ternary
  relations on the set $E$ of leaves of $T$.
  Choose two non-leaf children $x_1$ and $x_2$ of the root of $T$, and
  define $E_1$ and $E_2$ respectively as the leaves of the subtrees
  $T_1$ and $T_2$ of $T$ dandling from $x_1$ and $x_2$ respectively.
  Choose isomorphisms $\phi_1$ and $\phi_2$ from $E$ to $E_1$ and
  $E_2$ respectively. Define $E_{i,j}$ accordingly.

  Then, conditions (a) and (b) of
  Proposition~\ref{proposition.hopf} are satisfied, but not condition
  (c). Take indeed $a,b,c$ in $E_{1,1},E_{1,2},E_{2,2}$ respectively,
  and define $a',b',c':=\Phi(a,b,c)$ in $E_1, E_{2,1},E_{2,2}$
  respectively. The set $\{a,b,c\}$ has type $\ta$
  whereas the set $\{a',b',c'\}$ has type $\tb$.
  We recover the non-coassociative coproduct of the Planar Shuffle
  Algebra which splits the children of the root nodes in two
  consecutive ranges in all possible ways, and reduces the two
  resulting trees. For example:
  \begin{align*}
    \Delta(\leaf) & = \leaf \otimes 1 + 1 \otimes \leaf\\
    \Delta((\leaf,\leaf)) &= (\leaf,\leaf) \otimes 1 + \leaf \otimes \leaf + 1 \otimes (\leaf,\leaf)\\
    \Delta(\ta) & = \ta \otimes 1 + (\leaf,\leaf) \otimes \leaf + 1 \otimes \ta\,.
  \end{align*}
  Iterating the above, $\leaf\otimes\leaf\otimes\leaf$ appears in
  $(\Delta\otimes \id)(\Delta(\ta))$ but not in
  $(\id\otimes\Delta)(\Delta(\ta))$; that's the dual of the
  aforementioned counter example to (c).

  As far as we know, there currently is no known coassociative
  coproduct for this algebra, though it's likely to exist.
\end{example}

\subsection{Examples with a finite monomorphic decomposition}

\begin{example}
  \label{example.notfinitelygenerated2}%
  This example features another age algebra that is not finitely generated.

  Consider the relational structure $R:=(E,\rho)$, where
  $E:=\N \times \{1,2,3\}$ is endowed with the ternary relation
  $\rho:=\{ ( (i,1), (j,2), (k,3) ),\ i,j,k\in \N \}$. The minimal
  monomorphic decomposition is given by
  $(E_i:=\N\times \{i\})_{i\in X:=\{1,2,3\}}$. A basis of the age
  algebra $\agealgebra=\agealgebra(R)$ in degree $d$ is given by
  \begin{displaymath}
    x^\dvect, \quad \text{for $d_i>0,d_1+d_2+d_3=d$}
  \end{displaymath}
  together with
  \begin{displaymath}
    \sum_{\dvect\suchthat\text{$d_1=0$ or $d_2=0$ or $d_3=0$, $d_1+d_2+d_3=d$}} x^\dvect.
  \end{displaymath}
  The profile is given by $\phi_R(d)=\binom{d-1}2+1$ and the Hilbert
  series is $(\frac{x}{1-x})^3+\frac{1}{1-x}=\frac{x^3+x^2-2x+1}{(1-x)^3}$.

  By construction, the restriction of $E$ on $E_1\cup E_2$ is
  monomorphic. Therefore, the minimal monomorphic decomposition is not
  hereditary minimal and the age algebra is not finitely generated.

\end{example}

\begin{example}
  A variation of Example~A.19 of~\cite{Pouzet_Thiery.AgeAlgebra1}
  featuring a finitely generated age algebra with a Hilbert series
  whose numerator cannot be chosen with non-negative coefficients.
  This one has only two monomorphic parts which are both infinite.

  Let $R:=(E, \rho )$, where $E:=\N \times \{0,1\}$, $\rho:=[\N \times
  \{0\}]^3\cup [\N \times \{1\}]^3$. Then $R$ has two monomorphic
  parts, namely $\N \times \{0\}$ and $\N \times \{1\}$. Each
  type of $n$-element restriction has a representative made of a $m+k$
  element subset of $\N \times \{0\}$ and of a $m$-element subset of
  $\N \times \{1\}$ such that $n=2m+ k$; these representatives are
  non-isomorphic, except if $n=2$ (in the later case, all $2$ -element
  restrictions are isomorphic, hence we may eliminate the
  representative corresponding to $m=1, k=0$). With this observation,
  a straightforward computation shows that $\profile_R(0)=
  \profile_R(1)=\profile_R(2)=1$ and $\profile_R(n)=\lfloor \frac{n}{2}
  \rfloor +1$ for $n\geq 3$. Hence the generating series
  $H_{\profile_R}=\frac {1}{(1-Z)(1-Z^2)}- Z^2=\frac
  {1-Z^2+Z^3+Z^4-Z^5}{(1-Z)(1-Z^2)}$.

  But, then $H_{\profile_R}$ cannot be written as a quotient of the
  form $\frac{P}{(1-Z)(1-Z^k)}$ where $P$ is a polynomial with
  non-negative integer coefficients. Suppose indeed that
  $H_{\profile_R}$ is of this form. We may assume $k$ even
  (otherwise, multiply $P$ and $(1-Z)(1-Z^k)$ by $(1+Z^k$).  Set
  $k':=\frac{k}{2}$. Multiplying $1-Z^2+Z^3+Z^4-Z^5$ and
  $(1-Z)(1-Z^2)$ by $1+ Z^2 + \cdots+ Z^{2(k'-1)}$, we get
  $P=(1-Z^2+Z^3+Z^4-Z^5)(1+ Z^2 + \cdots+ Z^{2(k'-1)})$. Hence, the
  term of largest degree has a negative coefficient, a contradiction.
\end{example}

\begin{example}
  Another variation on Example~A.19 of~\cite{Pouzet_Thiery.AgeAlgebra1},
  with four variables; now the numerator can take either positive or
  negative coefficients.

  Let $R:=(E, (\rho, U_2,U_3) )$, where $E:=\N \times \{0,1,2,3\}$,
  $\rho:=\{((n,i),(m,j)): i=0, j\in \{1,2\}$ or $i=1,j=3\}$;
  $U_i:=\N\times \{i\}$ for $i\in \{2,3\}$. Then $R$ has four
  monomorphic components, namely $\N\times \{0\},\N \times \{1\},\N
  \times \{2\}, \N \times \{3\}$.  Let $S$ be the induced structure on
  four elements of the form $(x_i,i), i\in \{0,1,2,3\}$. A crucial
  property is that $S$ has only two non-trivial local isomorphisms,
  namely the map sending $(x_0,0)$ onto $(x_1,1)$ and its inverse.
  From this follows that the induced substructures on two $n$-element
  subsets $E$ are isomorphic if either they have the same number of
  elements on each $\N\times \{i\}$ or one subset is included into $\N
  \times \{0\}$, the other into $\N\times \{1\}$. Hence, the
  generating series $\hilbert_{\profile_R}$ is
  $\frac{1}{(1-Z)^4}-\frac{Z}{1-Z}= \frac
  {1-Z+3Z^2-3Z^3+Z^4}{(1-Z)^4}$. We may write it
  $\hilbert_{\profile_R}=\frac {Q_1}{(1-Z)(1-Z^4)(1-Z^5)(1-Z^5)}$
  where
  $Q_1:=1+2Z+6Z^2+10Z^3+14Z^4+17Z^5+18Z^6+14Z^7+10Z^8+6Z^9+Z^{10}$, as
  well as $\hilbert_{\profile_R}=\frac {Q}{(1-Z)(1-Z^5)^3}$ where
  $Q_2:=1+2Z+6Z^2+10Z^3+15Z^4+18Z^5+22Z^6+18Z^7+15Z^8+10Z^9+6Z^{10}+
  Z^{12}+Z^{16}$.
\end{example}

\subsection{Example with polynomial growth but infinitely many
  monomorphic parts}
\begin{examples}
  \label{example.notFinitelyGeneratedInfiniteDecomposition}
  Consider the direct sum
  $R':= \overline K_\infty \wr K_{1,1}\oplus R$ of the infinite
  matching and the relational structure $R$ of
  Example~\ref{example.notfinitelygenerated}. The profile has
  polynomial growth, but $R$ has infinitely many monomorphic parts and
  the age algebra is not finitely generated.
\end{examples}

\subsection{Miscellaneous example}
\begin{example}
  \label{example.infinite_profile}
  This example illustrates why some care needs to be taken when
  defining the age algebra for a relational structure with non finite
  profile.

  Take an infinite set $E$, with one binary relation for each couple
  $(i,j)$ of distinct elements of $E$, which holds just on $(i,j)$. In
  degree $1$, there is a single type; let $e_1$ be the corresponding
  element in the age algebra. Then, $e_1^2$ shall be the sum of all the
  types of degree $2$, of which there are infinitely many. Hence,
  for the age algebra to be indeed stable by multiplication, one shall
  consider infinite linear combinations of types.
\end{example}

\nocite{hodkinson_macpherson.1988}

\bibliographystyle{alpha}
\bibliography{main2}

\end{document}